\documentclass[11pt]{amsart}

\usepackage[foot]{amsaddr}
\usepackage{color,graphicx}

\usepackage{dsfont}
\usepackage{amssymb}
\usepackage{amsmath}
\usepackage{amsfonts}
\usepackage{graphicx}
\usepackage{amsthm}
\usepackage{enumerate}
\usepackage[mathscr]{eucal}
\usepackage{mathrsfs}
\usepackage{verbatim}
\usepackage{yhmath}

\calclayout

\usepackage{hyperref}
\hypersetup{colorlinks, linkcolor=blue}
\usepackage{soul}
\soulregister\cite7
\soulregister\eqref7
\soulregister\eqref7
\hypersetup{colorlinks, linkcolor=blue}

\newtheorem{thm}{Theorem}[section]
\newtheorem{lemma}[thm]{Lemma}
\newtheorem{prop}[thm]{Proposition}

\theoremstyle{definition}
\newtheorem{definition}[thm]{Definition}

\newtheorem{remark}[thm]{Remark}

\newcommand\ddd{\mathrm{d}}

\newcommand\supp{\mathrm{supp}}
\newcommand\holder{Hölder's~}
\newcommand\schrodinger{Schrödinger~}

\newcommand\rhu{\rightharpoonup}

\newcommand\bR{\mathbb{R}}

\newcommand\bN{\mathbb{N}}
\newcommand\bZ{\mathbb{Z}}

\def \l {\left}
\def \r {\right}

\makeatletter
\@namedef{subjclassname@2020}{
	\textup{2020} Mathematics Subject Classification}
\makeatother

\allowdisplaybreaks[4]

\begin{document}
	
\title[Odd profile decomposition]{On profile decomposition for Airy type equation}

\author{Boning Di}
\author{Chenjie Fan}
\author{Dunyan Yan}

\date{}
\thanks{Di was supported by the China Postdoctoral Science Foundation [Grant No. GZB20230812]. Fan was supported by the National Key R \& D Program of China [Grant No. 2021YFA1000800], the CAS Project for Young Scientists in Basic Research [Grant No. YSBR-031] and the National Natural Science Foundation of China [Grant No. 11688101]. Yan was supported by the National Natural Science Foundation of China [Grant Nos. 12071052 \& 12271501].}

\begin{abstract}
	We study the linear profile decomposition for the Airy type equation, where the associated Strichartz inequality corresponds to the Fourier extension inequality on the odd curve $\xi^{\ell}$. We also investigate an inhomogeneous case, modeled by the odd curve $\xi^3+\xi^5$ case. We note that, as observed by Frank and Sabin [Math. Ann., 2018], there is a two-profile phenomenon in the profile decomposition associated with odd curves.
\end{abstract}

\maketitle

\tableofcontents

\section{Introduction}
\subsection{The main statements}
Given an odd integer $\ell>1$ and the following homogeneous odd curve
\[L_{\ell}:=\l\{(\xi,\xi^{\ell}): \xi\in\bR\r\},\]
the associated Fourier extension estimate of \cite[Theorem 2.1]{KPV1991} states that
\begin{equation}\label{E:Odd Strichartz}
	\l\|[E_{\ell}]f\r\|_{L_t^6L_x^6(\bR^{2})} \leq \mathbf{M}_{\ell,q} \|f\|_{L_x^2(\bR)},
\end{equation}
where the extension operator $[E_{\ell}]$ is defined as
\[[E_{\ell}]f(x):=\frac{1}{2\pi} \int_{\bR} |\xi|^{\frac{\ell-2}{6}} e^{ix\xi+it\xi^{\ell}} \widehat{f}(\xi) \ddd \xi.\]
This Fourier extension type estimate is also known as Strichartz type estimate. For convenience, we denote the Strichartz norms $\|F\|_{L_{t,x}^{p}}:= \|F\|_{L_t^p L_x^p}$.

This article aims to establish the profile decomposition associated with the Fourier extension inequality \eqref{E:Odd Strichartz}, and then study the extremal sequences for this inequality. To investigate the scaling-broken situation, this paper is also devoted to an inhomogeneous odd curve $\xi^3+\xi^5$ case.

One can observe that both sides of \eqref{E:Odd Strichartz} are invariant under some symmetries, and we will particularly investigate the noncompact symmetries below. The \textit{symmetries} $[g_n]$ for the Strichartz inequality \eqref{E:Odd Strichartz} are scaling and time-space translations as follows
\[[g_n]f(x):=[e^{it_n \nabla^{\ell}}] \l((h_n)^{-1/2} f\l(\frac{x-x_n}{h_n}\r)\r), \quad (h_n,x_n,t_n)\in \bR_{+}\times \bR\times\bR,\]
where the operator $[e^{it \nabla^{\ell}}]$ is defined by
\[[e^{it \nabla^{\ell}}]f(x):=\frac{1}{2\pi} \int_{\bR} e^{ix\xi+it\xi^{\ell}} \widehat{f}(\xi) \ddd \xi;\]
and the \textit{associated group} $G$ is the following
\begin{equation} \label{E: Associated group}
  G:=\Big\{[g_n]: (h_n,x_n,t_n)\in \bR_{+}\times \bR\times\bR\Big\}.  
\end{equation}
To state our profile decomposition result, for parameters
\[(h_n^j, x_n^j, \xi_n^j, t_n^j) \in \bR_{+} \times \bR \times \bR \times \bR,\]
we need to introduce the profile operator $[T_n^j]$ as follows
\begin{equation} \label{E:Bilinear operator}
	[T_n^j](\phi_1, \phi_2) (x):= [T_{n,+}^{j}] \phi_1 + [T_{n,-}^{j}] \phi_2,
\end{equation}
where the operators $[T_{n,+}^{j}]$ and $[T_{n,-}^{j}]$ are defined as
\begin{equation*}
	[T_{n,\pm}^{j}] \phi:= [e^{it_n^j}\nabla^{\ell}] \l[(h_n^j)^{-\frac{1}{2}} e^{\pm i(x-x_n^j)\xi_n^j} \phi \l(\frac{x-x_n^j}{h_n^j}\r)\r].
\end{equation*}
\begin{definition}[Orthogonal parameters: odd curve] \label{D:Odd orthogonality}
	We say the parameters $(h_n^j, x_n^j, \xi_n^j, t_n^j)$ are \textit{odd pairwise orthogonal} if for arbitrary fixed $j\neq k$, these parameters satisfy one of the following:
	\begin{itemize}
		\item $(h_n^j,\xi_n^j) \neq (h_n^k,\xi_n^k)$ with
		\[\lim_{n\to \infty}\l(\frac{h_n^j}{h_n^k} + \frac{h_n^k}{h_n^j} + (h_n^j+h_n^k) |\xi_n^j-\xi_n^k| \r) = \infty\]
		and
		\[\lim_{n\to \infty}\l(\frac{h_n^j}{h_n^k} + \frac{h_n^k}{h_n^j} + (h_n^j+h_n^k) |\xi_n^j+\xi_n^k| \r) = \infty;\]
		\item $(h_n^j,\xi_n^j) \equiv (h_n^k,\xi_n^k) \equiv (h_n, \xi_n)$ with $h_n\xi_n \to +\infty$ and
		\[\lim_{n\to\infty} \l(\l|\frac{x_n^j-x_n^k}{h_n} - \frac{(t_n^j-t_n^k)\alpha (h_n \xi_n)^{\alpha-1}}{(h_n)^{\alpha}}\r| + \l|\frac{(t_n^j-t_n^k) (h_n \xi_n)^{\alpha-2}}{(h_n)^{\alpha}}\r| \r)= \infty;\]
		\item $(h_n^j,\xi_n^j) \equiv (h_n^k,\xi_n^k) \equiv (h_n, \xi_n)$ with $h_n\xi_n \to -\infty$ and
		\[\lim_{n\to\infty} \l(\l|\frac{x_n^j-x_n^k}{h_n} - \frac{(t_n^j-t_n^k)\alpha (h_n \xi_n)^{\alpha-1}}{(h_n)^{\alpha}}\r| + \l|\frac{(t_n^j-t_n^k) (h_n \xi_n)^{\alpha-2}}{(h_n)^{\alpha}}\r| \r)= \infty;\]
		\item $(h_n^j,\xi_n^j) \equiv (h_n^k,\xi_n^k) \equiv (h_n, \xi_n)$ with $\xi_n \equiv 0$ and
		\[\lim_{n\to\infty} \l(\l|\frac{x_n^j-x_n^k}{h_n} \r| + \l|\frac{t_n^j-t_n^k}{(h_n)^{\alpha}}\r| \r)= \infty.\]
	\end{itemize}
\end{definition}

\begin{thm}[Profile decomposition: odd curve] \label{T:Profile decomposition}
	Let $(f_n)$ be a bounded sequence in $L^2(\bR)$. Then, up to subsequences, there exist a sequence of operators $([T_n^j])$ defined by \eqref{E:Bilinear operator} with odd pairwise orthogonal parameters $(h_n^j, x_n^j, \xi_n^j, t_n^j)$ and a sequence of functions $(\phi_{+}^j, \phi_{-}^j) \subset L^2(\bR) \times L^2(\bR)$ such that for every integer $J\in \bN$, we have the profile decomposition
	\begin{equation}\label{T:Profile decomposition-1}
		f_n=\sum_{j=1}^{J} [T_n^j](\phi_{+}^j, \phi_{-}^j) +\omega_n^{J},
	\end{equation}
	where this decomposition possesses the following properties: firstly, the remainder term $\omega_n^{J}$ has a vanishing Strichartz norm
	\begin{equation} \label{T:Profile decomposition-2}
		\lim_{J\to\infty}\limsup_{n\to\infty}\l\|[E_{\ell}] \omega_n^{J} \r\|_{L_{t,x}^{6}}=0;
	\end{equation}
	secondly, for each $j\leq J$, there holds the following $L_x^2$ weak convergence
	\begin{equation} \label{T:Profile decomposition-3}
		[T_{n,+}^j]^{-1} \omega_n^J \rightharpoonup 0 \quad \text{and} \quad [T_{n,-}^j]^{-1} \omega_n^J \rightharpoonup 0 \quad \text{as} \quad n\to \infty.
	\end{equation}
\end{thm}

\begin{remark}[Odd-orthogonality] \label{R:Profile decomposition-Strichartz orthogonality}
	One can compare the odd-orthogonality in Definition \ref{D:Odd orthogonality} and the even-orthogonality in Definition \ref{D:Even orthogonality}. As stated in Remark \ref{R:Profile decomposition-even-1}, this odd-orthogonality of parameters is essentially equivalent to the following fact: the operators $[T_n^j]$ are orthogonal in the sense that for arbitrary fixed $j\neq k$ there holds
	\begin{equation}\label{T:Profile decomposition-4}
		[T_{n,+}^k]^{-1} [T_{n,+}^j]\rightharpoonup 0 \quad \text{and} \quad [T_{n,+}^k]^{-1} [T_{n,-}^j]\rightharpoonup 0 \quad \text{as} \quad n\to \infty
	\end{equation}
	in the weak operator topology of $\mathcal{B}(L^2)$. Hence by \eqref{T:Profile decomposition-3}, for each $J\in\bN$, there holds the $L^2$-limit orthogonality
	\begin{equation}\label{T:Profile decomposition-5}
		\lim_{n\to\infty}\l[\|f_n\|_{L_x^2}^2- \l(\sum_{j=1}^{J} \l\|\phi_{+}^j \r\|_{L_x^2}^2 + \l\|\phi_{-}^j \r\|_{L_x^2}^2\r)-\|\omega_n^{J}\|_{L_x^2}^2\r]=0.
	\end{equation}
	Moreover, the odd-orthogonality of parameters also implies that
	\begin{equation}\label{T:Profile decomposition-6}
		\lim_{n\to\infty} \l\|[E_{\ell}] [T_n^j](\phi_{+}^j,\phi_{-}^j) \cdot [E_{\ell}] [T_n^k] (\phi_{+}^k, \phi_{-}^k) \r\|_{L_{t,x}^3}=0 \quad \forall j\neq k,
	\end{equation}
	which further implies that for each $J\in\bN$ there holds the following Strichartz-limit orthogonality
	\begin{equation}\label{T:Profile decomposition-7}
        \limsup_{n\to\infty}\l(\l\|\sum_{j=1}^J[E_{\ell}] [T_n^j](\phi_{+}^j, \phi_{-}^j)\r\|_{L_{t,x}^6}^{6} -\sum_{j=1}^J\l\|[E_{\ell}] [T_n^j](\phi_{+}^j, \phi_{-}^j) \r\|_{L_{t,x}^{6}}^{6}\r)=0.
	\end{equation}
    Finally, for each fixed $j$, we can further assume the parameters $(h_n^j,\xi_n^j)$ satisfy either $h_n^j\xi_n^j \to +\infty$, $h_n^j\xi_n^j \to -\infty$ or $\xi_n^j \equiv 0$. 
\end{remark}

\begin{remark}[Real version] \label{R:Profile decomposition-real version}
	For the real-valued sequence $(f_n)$, one can obtain a similar profile decomposition consequence by taking the real parts on both sides of \eqref{T:Profile decomposition-1}. Moreover, by the identity
	\[\mathrm{Re} [T_n^j](\phi_{+}^j, \phi_{-}^j) = [T_{n,+}^j] \l(\frac{\phi_{+}^j + \overline{\phi}_{-}^j}{2}\r) + [T_{n,-}^j] \l(\frac{\overline{\phi}_{+}^j + \phi_{-}^j}{2}\r),\]
	one can write the real version profile decomposition as
	\[f_n=\sum_{j=1}^{J} [T_n^j](\phi^j, \overline{\phi}^j) +\omega_n^{J}, \quad \phi^j:= \frac{\phi_{+}^j + \overline{\phi}_{-}^j}{2}.\]
    Hence, for the real version, there is essentially only one profile $\phi^j$ instead of two profiles $(\phi_{+}^j,\phi_{-}^j)$.
\end{remark}

One can see that there is a two-profile phenomenon $(\phi_{+}^j, \phi_{-}^j)$ in our profile decomposition result Theorem \ref{T:Profile decomposition}. For these two reduced profiles with parameters $|h_n\xi_n|\to\infty$, the corresponding two-profile Strichartz norm asymptotic behavior has been studied by Frank and Sabin \cite[Lemma 2.4 \& Proof of Theorem 2]{FS2018}, where they realize a homogenization phenomenon for the case $\ell=3$. Their consequence can be parallelly extended to the general odd integer case $\ell \geq 3$, which we state as follows.
\begin{prop}[Two-profile norm asymptotic \cite{FS2018}] \label{P:Two-profile limit-odd}
	For any function $f\in L^2(\bR)$, denoting
	\[\widetilde{f}_N(x):= e^{ix N} f(x) + e^{-ix N} \overline{f}(x),\]
	then there holds
	\[\lim_{N\to\infty} \l\|[E_{\ell}]\widetilde{f}_N \r\|_{L_{t,x}^6} = \l[\frac{40}{\ell(\ell-1)}\r]^{1/6} \l\|[e^{it\Delta}]f\r\|_{L_{t,x}^6}.\]
	Moreover, for arbitrary two functions $g_1$ and $g_2$ in $L^2(\bR)$, denoting
	\[\widetilde{g}_N(x):= e^{ix N} g_1(x) + e^{-ix N} \overline{g}_2(x),\]
	then there holds
	\[\lim_{N\to\infty} \l\|[E_{\ell}]\widetilde{g}_N \r\|_{L_{t,x}^6} \leq \l[\frac{5}{\ell(\ell-1)}\r]^{1/6} \l(\l\|[e^{it\Delta}]g_1\r\|_{L_{t,x}^6}^2 + \l\|[e^{it\Delta}]g_2\r\|_{L_{t,x}^6}^2 \r)^{1/2},\]
	and the equality holds if and only if
	\[\l|[e^{it\Delta}]g_1(x)\r|= \l|[e^{it\Delta}]g_2(x)\r| ,\quad \text{a.e.} \;\; (t,x) \in \bR^2.\]
\end{prop}

Later in Proposition \ref{P:Two-profile limit-3,5}, we will further investigate this two-profile asymptotic behavior for the inhomogeneous odd curve $\xi^3+\xi^5$ case. Indeed, the profile decomposition Theorem \ref{T:Profile decomposition} and the two-profile asymptotic Proposition \ref{P:Two-profile limit-odd} can imply a precompactness result for the aforementioned Fourier extension inequality \eqref{E:Odd Strichartz}. Here we state this precompactness result as Theorem \ref{T:Precompactness}.

\begin{definition}[Extremal terminologies]
	Recall the group $G$ of symmetries defined in \eqref{E: Associated group}. We say a sequence of functions $(f_n)$ in $L^2(\bR)$ is \textit{precompact up to symmetries} if there exists a sequence of symmetries $([g_n])$ in $G$ such that $\l([g_n]f_n\r)$ has convergent subsequence in $L^2(\bR)$. Meanwhile, a sequence of functions $(f_n)$ in $L^2(\bR)$ is an \textit{extremal sequence} for $\mathbf{M}_{\ell}$ if it satisfies
	\[\|f_n\|_{L_x^2(\bR)}=1,\quad \lim_{n\to\infty} \l\|[E_{\ell}] f_n \r\|_{L_t^6L_x^6(\bR^2)}=\mathbf{M}_{\ell}.\]
	Furthermore, a function $f_{*}$ in $L^2(\bR)$ is called an \textit{extremal function} for $\mathbf{M}_{\ell}$ if $f_{*}$ can make the inequality \eqref{E:Odd Strichartz} an equality and $\|f_{*}\|_{L_x^2}=1$.
\end{definition}

\begin{thm}[Precompactness: odd curve] \label{T:Precompactness}
	All extremal sequences for $\mathbf{M}_{\ell}$ are precompact up to symmetries if and only if
	\begin{equation} \label{T:Precompactness-1}
		\mathbf{M}_{\ell}> \l[\frac{5}{\ell(\ell-1)}\r]^{1/6} \mathbf{M}_{2}.
	\end{equation}
	Here $\mathbf{M}_{2}$ is the sharp constant for the Schr\"odinger operator defined by
	\[\mathbf{M}_{2}:= \sup \l\{\l\|[e^{it\Delta}]f \r\|_{L_{t,x}^{6}}: \|f\|_{L_x^2}=1 \r\}.\]
	In particular, if the strict inequality \eqref{T:Precompactness-1} holds, then there exists an extremal for $\mathbf{M}_{\ell}$.
\end{thm}

\begin{remark}
    For the even curve $|\xi|^{\ell}$ case studied in \cite{BOQ2020,DY2023,JPS2010}, there is no such two-profile phenomenon mentioned above, and thus the corresponding constant in \eqref{T:Precompactness-1} is a smaller canstant $[\ell(\ell-1)/2]^{-1/6}$ rather than $[\ell(\ell-1)/5]^{-1/6}$.
\end{remark}

Next, we investigate the inhomogeneous odd curve case
\[\widetilde{L}_{5} := \l\{(\xi,\xi^{3} + \xi^5): \xi\in\bR\r\}.\]
Here we should emphasize that $\xi^{3} + \xi^5$ is a model case, which is scaling broken. The associated Fourier extension estimate of \cite[Theorem 2.1 \& Remark (c)]{KPV1991} states that
\begin{equation}\label{E:Odd Strichartz-3,5}
	\l\|[\widetilde{E}_{5}]f\r\|_{L_{t,x}^6 (\bR^{2})} \leq \widetilde{\mathbf{M}}_{5} \|f\|_{L_x^2(\bR)},
\end{equation}
where the extension operator $[\widetilde{E}_{5}]$ is defined as
\[[\widetilde{E}_{5}]f(x):=\frac{1}{2\pi} \int_{\bR} |\xi + \xi^3|^{\frac{1}{6}} e^{ix\xi+it(\xi^3+\xi^5)} \widehat{f}(\xi) \ddd \xi.\]
Here we similarly define the \textit{symmetries} $[\widetilde{g}_n]$ for the Strichartz inequality \eqref{E:Odd Strichartz-3,5} by
\[[\widetilde{g}_n]f(x):=\l[e^{it_n (\nabla^3 + \nabla^5)} \r] \l((h_n)^{-1/2} f\l(\frac{x-x_n}{h_n}\r)\r)\]
with
\[\l[e^{it_n (\nabla^3 + \nabla^5)} \r] f(x):=\frac{1}{2\pi} \int_{\bR} e^{ix\xi+it(\xi^3+\xi^5)} \widehat{f}(\xi) \ddd \xi,\]
and the \textit{associated group} $\widetilde{G}$, as well as the corresponding extremal terminologies. For parameters $(h_n^j, x_n^j, \xi_n^j, t_n^j) \subset \bR_{+} \times \bR^3$, we similarly introduce the profile operator $[\widetilde{T}_n^j]$ as follows
\begin{equation} \label{E:Bilinear operator-3,5}
	[\widetilde{T}_n^j](\phi_1, \phi_2) (x):= [\widetilde{T}_{n,+}^{j}]\phi_1 + [\widetilde{T}_{n,-}^{j}]\phi_2,
\end{equation}
where the operators $[\widetilde{T}_{n,+}^{j}]$ and $[\widetilde{T}_{n,-}^{j}]$ are defined as
\begin{equation*}
	[\widetilde{T}_{n,\pm}^{j}] \phi:= \l[e^{it_n^j(\nabla^3+ \nabla^5)}\r] \l[(h_n^j)^{-\frac{1}{2}} e^{\pm i(x-x_n^j)\xi_n^j} \phi \l(\frac{x-x_n^j}{h_n^j}\r)\r].
\end{equation*}

\begin{definition}[Orthogonal parameters: inhomogeneous odd curve]
	We say the parameters $(h_n^j, x_n^j, \xi_n^j, t_n^j)$ are \textit{inhomogeneous odd pairwise orthogonal} if for arbitrary fixed $j\neq k$, these parameters satisfy one of the following:
	\begin{itemize}
		\item $(h_n^j,\xi_n^j) \neq (h_n^k,\xi_n^k)$ with
		\[\lim_{n\to \infty}\l(\frac{h_n^j}{h_n^k} + \frac{h_n^k}{h_n^j} + (h_n^j+h_n^k) |\xi_n^j-\xi_n^k| \r) = \infty\]
		and
		\[\lim_{n\to \infty}\l(\frac{h_n^j}{h_n^k} + \frac{h_n^k}{h_n^j} + (h_n^j+h_n^k) |\xi_n^j+\xi_n^k| \r) = \infty;\]
		\item $(h_n^j,\xi_n^j) \equiv (h_n^k,\xi_n^k) \equiv (h_n, \xi_n)$ with $h_n\xi_n \to +\infty$ and
		\[\lim_{n\to\infty} \l(\l|\frac{x_n^j-x_n^k}{h_n} -\frac{3(t_n^j-t_n^k) (h_n \xi_n)^{2}}{(h_n)^{3}} - \frac{5(t_n^j-t_n^k) (h_n \xi_n)^{4}}{(h_n)^{5}}\r| + \l|\frac{(t_n^j-t_n^k) (h_n \xi_n)^{3}}{(h_n)^{5}}\r| \r)= \infty;\]
		\item $(h_n^j,\xi_n^j) \equiv (h_n^k,\xi_n^k) \equiv (h_n, \xi_n)$ with $h_n\xi_n \to -\infty$ and
		\[\lim_{n\to\infty} \l(\l|\frac{x_n^j-x_n^k}{h_n} -\frac{3(t_n^j-t_n^k) (h_n \xi_n)^{2}}{(h_n)^{3}} - \frac{5(t_n^j-t_n^k) (h_n \xi_n)^{4}}{(h_n)^{5}}\r| + \l|\frac{(t_n^j-t_n^k) (h_n \xi_n)^{3}}{(h_n)^{5}}\r| \r)= \infty;\]
		\item $(h_n^j,\xi_n^j) \equiv (h_n^k,\xi_n^k) \equiv (h_n, \xi_n)$ with $\xi_n \equiv 0$ and
		\[\lim_{n\to\infty} \l(\l|\frac{x_n^j-x_n^k}{h_n} \r| + \l|\frac{t_n^j-t_n^k}{(h_n)^{3}}\r| + \l|\frac{t_n^j-t_n^k}{(h_n)^{5}}\r|\r)= \infty.\]
	\end{itemize}
\end{definition}

\begin{thm}[Profile decomposition: inhomogeneous odd curve] \label{T:Profile decomposition-3,5}
	Let $(f_n)$ be a bounded sequence in $L^2(\bR)$. Then, up to subsequences, there exist a sequence of operators $([\widetilde{T}_n^j])$ defined by \eqref{E:Bilinear operator-3,5} with inhomogeneous odd pairwise orthogonal parameters $(h_n^j, x_n^j, \xi_n^j, t_n^j)$ and a sequence of functions $(\phi_{+}^j, \phi_{-}^j) \subset L^2(\bR) \times L^2(\bR)$ such that for every integer $J\in \bN$, we have the profile decomposition
	\begin{equation}\label{T:Profile decomposition-3,5-1}
		f_n=\sum_{j=1}^{J} [\widetilde{T}_n^j](\phi_{+}^j, \phi_{-}^j) +\omega_n^{J},
	\end{equation}
	where this decomposition possesses the following properties: firstly, the remainder term $\omega_n^{J}$ has vanishing Strichartz norm
	\begin{equation}\label{T:Profile decomposition-3,5-2}
		\lim_{J\to\infty}\limsup_{n\to\infty}\l\|[\widetilde{E}_{5}] \omega_n^{J} \r\|_{L_{t,x}^{6}}=0;
	\end{equation}
	secondly, for each $j\leq J$, there holds the following $L_x^2$ weak convergence
	\begin{equation} \label{T:Profile decomposition-3,5-3}
		[\widetilde{T}_{n,+}^j]^{-1} \omega_n^J \rightharpoonup 0 \quad \text{and} \quad [\widetilde{T}_{n,-}^j]^{-1} \omega_n^J \rightharpoonup 0 \quad \text{as} \quad n\to \infty.
	\end{equation}
\end{thm}

\begin{remark} \label{R:Profile decomposition-3,5}
	As stated in Remark \ref{R:Profile decomposition-Strichartz orthogonality}, the orthogonality of parameters is essentially equivalent to the fact that for arbitrary fixed $j\neq k$ there holds
	\begin{equation}\label{T:Profile decomposition-3,5-4}
		[\widetilde{T}_{n,+}^k]^{-1} [\widetilde{T}_{n,+}^j]\rightharpoonup 0 \quad \text{and} \quad [\widetilde{T}_{n,+}^k]^{-1} [\widetilde{T}_{n,-}^j]\rightharpoonup 0 \quad \text{as} \quad n\to \infty
	\end{equation}
	in the weak operator topology of $\mathcal{B}(L^2)$; then for each $J\geq1$, there holds the $L^2$-limit orthogonality
	\begin{equation}\label{T:Profile decomposition-3,5-5}
		\lim_{n\to\infty}\l[\|f_n\|_{L_x^2}^2- \l(\sum_{j=1}^{J} \l\|\phi_{+}^j \r\|_{L_x^2}^2 + \l\|\phi_{-}^j \r\|_{L_x^2}^2\r)-\|\omega_n^{J}\|_{L_x^2}^2\r]=0.
	\end{equation}
	and for each $J\geq 1$, there holds the Strichartz-limit orthogonality
	\begin{equation}\label{T:Profile decomposition-3,5-6}
        \limsup_{n\to\infty}\l(\l\|\sum_{j=1}^J[\widetilde{E}_{\ell}] [\widetilde{T}_n^j](\phi_{+}^j, \phi_{-}^j)\r\|_{L_{t,x}^6}^{6} -\sum_{j=1}^J\l\|[\widetilde{E}_{\ell}] [\widetilde{T}_n^j](\phi_{+}^j, \phi_{-}^j) \r\|_{L_{t,x}^{6}}^{6}\r)=0.
	\end{equation}
        For each fixed $j$, we can further assume the parameters $(h_n^j,\xi_n^j)$ satisfy either $h_n^j\xi_n^j \to \pm\infty$ or $\xi_n^j \equiv 0$. Furthermore, as shown in Remark \ref{R:Profile decomposition-real version}, one can similarly obtain the real version profile decomposition for the inhomogeneous odd curve case.
\end{remark}

Finally, similar to the homogeneous odd curve case, combining the inhomogeneous odd curve version profile decomposition Theorem \ref{T:Profile decomposition-3,5} and two-profile Strichartz norm asymptotic Proposition \ref{P:Two-profile limit-3,5}, one can obtain the following precompactness result. For simplicity, we omit its proof.
\begin{thm}[Precompactness: inhomogeneous odd curve] \label{T:Precompactness-3,5}
	All extremal sequences for $\widetilde{\mathbf{M}}_5$ are precompact up to symmetries if and only if
	\begin{equation} \label{T:Precompactness-3,5-1}
		\widetilde{\mathbf{M}}_{5}> 2^{-1/3} \mathbf{M}_{2}.
	\end{equation}
	In particular, if the strict inequality \eqref{T:Precompactness-3,5-1} holds, then there exists an extremal for $\widetilde{\mathbf{M}}_{5}$.
\end{thm}

\subsection{Background}
Profile decomposition plays an important role in the study of various variational problems. In particular, it is an essential part of the concentration-compactness-rigidity method (also called as Kenig-Merle road map) \cite{BG1999,Gerard1998,KM2006,KM2008}, which is very popular in the study of nonlinear dispersive PDEs. It is also a very useful tool in the study of extremal variational type problems.

To illustrate the idea and explain the situation, we use the 1d linear \schrodinger equation as an example.

Consider the following linear \schrodinger equation
\begin{equation}\label{E:1d Schrodinger}
	iv_{t}+\Delta v=0,\quad v(0,x)=v_{0}.	
\end{equation}
One fundamental inequality for the linear propagator $e^{it\Delta}$ to \eqref{E:1d Schrodinger} is the Strichartz inequality
\begin{equation}\label{E:1d Strichartz}
	\|e^{it\Delta}f\|_{L_{t,x}^{6}}\leq  \mathbf{M}_2 \|f\|_{L_{x}^{2}}.
\end{equation}
The notion of profile decomposition rises because of the inherent symmetries of inequality of \eqref{E:1d Strichartz}. One notes that \eqref{E:1d Strichartz} is invariant under time translation, space translation, Galilean boost, and scaling. Here scaling means if one plugs $f_{\lambda}:=\lambda^{-\frac{1}{2}}f(\cdot/\lambda)$ into \eqref{E:1d Strichartz}, both sides of \eqref{E:1d Strichartz} stays the same. Those symmetries cause a lack of compactness of Strichartz inequality, (which may be thought as a dispersive version of Sobolev embedding), and in particular cause obstacles in the study of variation problems.

For example, let us consider the extremal problem for the Strichartz inequality \eqref{E:1d Strichartz}. To be specific, if one considers an extremal sequence $(f_{n})$ for the inequality \eqref{E:1d Strichartz}. In general, one cannot expect $(f_{n})$ to be precompact in $L^{2}$. Because, for example, for any sequence $(x_{n})$, if one poses $g_{n}=f_{n}(x-x_{n})$, then $(g_{n})$ is also an extremal sequence.

Here, a typical profile decomposition argument will proceed as the following. Given such an extremal sequence $(f_{n})$, one can find
\[(\phi^j) \subset L^2(\bR), \quad (h_n^j,\xi_n^j,t_n^j,x_n^j) \subset \bR_{+}\times \bR^3\]
such that up to picking subsequence, one has the following decomposition
\begin{equation}\label{eq: profiled1}
	f_{n}=\sum_{j=1}^{J}[T_n^j]\phi^j + \omega_{n}^J	
\end{equation}
with the vanishing remainder term
\begin{equation}\label{eq:vanish}
	\lim_{J\to\infty}\limsup_{n\to\infty} \l\|e^{it\Delta} \omega_n^J \r\|_{L_{t,x}^{6}}=0;
\end{equation}
and for each fixed $j\leq J$ the $L_x^2$ weak convergence relation
\begin{equation}\label{eq: wc}
	[T_n^j]^{-1} \omega_{n}^J \rightharpoonup 0;
\end{equation}
and the key $\mathcal{B}(L^2)$ weak operator topology orthogonality condition
\begin{equation}\label{eq: ortho}
	[T_n^j]^{-1} [T_n^k] \rightharpoonup 0, \quad \forall j\neq k.
\end{equation}
Here the equation \eqref{eq: profiled1} indeed should be read from right to left, and in particular it defines $\omega_n^J$. Note that the profiles $\phi^j$ themselves are not orthogonal in any sense, and it is totally possible $\phi^1=\phi^2$. Also note that \eqref{eq: wc} and \eqref{eq: ortho} allow one to extract profiles by studying weak convergence and $\phi^j$ is defined by the weak limit of $[T_n^j]^{-1}f_{n}$. For example, if one considers the sequence
\[f_{n}(x)=f(x)e^{inx}+f(x)e^{-inx},\]
one has
\[\phi^1=\phi^2=f,\quad \phi^j \equiv 0 \;\; \text{for} \;\; j\geq 3, \quad \xi_{n}^1=n,\quad \xi_{n}^2=-n,\quad h_n^j\equiv t_n^j \equiv x_n^j \equiv 0 \;\; \text{for} \;\; j\in\bN.\]
What is crucial for profile decomposition to be useful in the study of extremal sequence $(f_{n})$ is that \eqref{eq: ortho} ensures
\begin{equation}\label{eq: bilinear}
	\lim_{n\to\infty} \l\|[e^{it\Delta}] [T_n^j]\phi^j \cdot [e^{it\Delta}] [T_n^k] \phi^k \r\|_{L_{t,x}^{3}}=0.
\end{equation}
And also by \eqref{eq: wc} and \eqref{eq: ortho}, for each fixed $J$, one has 
\begin{equation}\label{eq: masssum}	
	\lim_{n\rightarrow \infty}\sum_{j=1}^{J}\|\phi^j\|_{L_x^2}^{2}+\|\omega_{n}^J\|_{L_x^2}^{2}-\|f_{n}\|_{L_x^2}^{2}=0.
\end{equation}
Property \eqref{eq: masssum} is equally important, but let us highlight \eqref{eq: bilinear} for the moment. Property \eqref{eq: bilinear} means that profiles at different scales essentially decouple with each other. Hence, up to some direct calculus, one concludes if $(f_{n})$ is an extremal sequence, then there is at most one nontrivial profile $\phi^1$, and one can further conclude that $\phi^1$ is an extremal function for \eqref{E:1d Strichartz}.

To summarize, in the discussion above, it is very important that different profiles decouple with each other in some sense, i.e. \eqref{eq: bilinear}, and this decoupling phenomenon comes from the parameters of $[T_n^j]$ but not the profiles themselves. And, in practice, only one nontrivial profile exists in the study of our interested sequence. This is consistent with the general heuristic that an optimal/minimal object in a variational problem cannot be split into two non-trivial pieces.

Profile decompositions based on Strichartz type estimates are quite general, and also work for wave equation, Klein-Gordon equation (note this model is in some sense inhomogeneous and does not satisfy the full scaling symmetry), etc.

One will expect a parallel theory for the Airy equation
\begin{equation}\label{eq: airy}
	u_{t}+u_{xxx}=0
\end{equation}
and the relevant Strichartz estimate
\begin{equation}\label{eq: as}
	\l\|[D^{\frac{1}{6}}] [e^{it\nabla^{3}}] f \r\|_{L_{t,x}^{6}(\bR^2)} \leq C\|f\|_{L_{x}^{2}(\bR)}.
\end{equation}
However, it has been observed by Frank-Sabin \cite{FS2018} that a parallel theory for \eqref{eq: airy} is simply wrong due to a two-profile phenomenon. To be more specific, one may just consider a specific sequence 
\begin{equation} \label{E:Two profiles function}
	f_{n}(x):=\phi(x)e^{inx}+\phi(x)^{-inx}\equiv f_{1,n}+f_{2,n},
\end{equation}
where $\phi$ is some nice Schwartz function on $\mathbb{R}$. We assume $\phi$ is real valued for simplicity. Following the Schr\"odinger case, one may expect that 
\begin{equation}\label{eq: wishlist}
	\lim_{n\rightarrow \infty} \l\|[D^{\frac{1}{6}}] [e^{it\nabla^{3}}] f_{1,n} \cdot [D^{\frac{1}{6}}] [e^{it\nabla^{3}}] f_{2,n} \r\|_{L_{t,x}^{3}}=0.
\end{equation}
However, observe that
\begin{equation}\label{eq: onemoresymmetry}
	f_{2,n}=\overline{f_{1,n}},\quad  [D^{\frac{1}{6}}] [e^{it\nabla^{3}}] f_{2,n}=\overline{[D^{\frac{1}{6}}] [e^{it\nabla^{3}}] f_{1,n}}.
\end{equation}
One see \eqref{eq: wishlist} cannot hold. The failure of \eqref{eq: wishlist}, at first glance, makes one wonder about the usefulness of profile decomposition in the study of \eqref{eq: airy}, for example, the study of extremal sequence for \eqref{eq: as}.

In fact, the question regarding the extremal sequence of \eqref{eq: as} has been studied by Frank-Sabin \cite{FS2018}. Among other important contributions, Frank and Sabin realize a homogenization phenomenon when one studies the Strichartz norm $\|[D^{1/6}] [e^{it\nabla^{3}}] f_n\|_{L_{t,x}^{6}}$ with the functions $f_n$ coming form \eqref{E:Two profiles function}. They have established a two-profile Strichartz norm asymptotic consequence, stated in Proposition \ref{P:Two-profile limit-odd} above, by applying their homogenization result \cite[Lemma B.1]{FS2018}, which is essentially an extension of \cite[Lemma 5.2]{Allaire1992}. Hence, their effort recovers the usefulness of profile decomposition theory in the study of extremal sequences for \eqref{eq: as}. It remains an interesting question to study the parallel theory on complex-valued mass critical gKdV.

Here we briefly sketch some of the works on extremals for Fourier extension type inequalities. These problems have been studied extensively in the last two decades. Works in this research area include abundant situations such as spheres \cite{BBI2015,CFOT2017,CG2022,CO2015,COS2019Adv,CS2012A&P,CS2012Adv,FLS2016,Foschi2015,FS2022,GN2022,MO2022,OQ2021,OTZ2022,Shao2016}, paraboloids \cite{BBCH2009,Carneiro2009,DMR2011,Foschi2007,FS2022DPDE,GN2022,GZ2022,Goncalves2019,Goncalves2019JFA,HZ2006,Kunze2003,Shao2009EJDE,Tautges2022,WZ2021} cones \cite{Carneiro2009,Foschi2007,GN2022,NOST2023,Quilodran2013,Ramos2012,OR2013}, fractional surfaces \cite{BOQ2020,DY2023,DY2022b,JPS2010,JSS2017,OQ2018}, odd curves \cite{BOQ2020,FS2018,Shao2009}, hyperboloids \cite{COS2019ANIHPC,COSS2021,Quilodran2015,Quilodran2022}, nonendpoint type inequalities \cite{FVV2011,FVV2012,HS2012}, general $L^q$-type inequalities \cite{BS2023,BS2023Arxiv,CQ2014,FS2022,NOST2023,Stovall2020} and so on. For a more comprehensive discussion of recent progress, readers are referred to the survey papers \cite{FO2017} and \cite{NOT2023}.

Let us now come back to the aforementioned Airy equation situation. It should be further noted that the aforementioned sequence $\phi(x)e^{inx}+\phi(x) e^{-inx}$ is essentially the only example such that a classical profile decomposition fails. Therefore, one may expect to use the profile decomposition for even type operator $[e^{it|\nabla|^{\ell}}]$ to establish the desired profile decomposition for odd type operator $[e^{it\nabla^{\ell}}]$. The first main point of this paper is to achieve this idea by using some divide-and-regroup arguments, which are shown in Section \ref{S:Homogeneous odd curve}. Roughly speaking, we first do classical profile decomposition on the (Fourier space) positive part and negative part respectively, then regroup the profiles appropriately to establish our desired odd curve version profile decomposition.

The second main point of this paper is to study the inhomogeneous odd curve $\xi^3+\xi^5$ case: the profile decomposition and the two-profile Strichartz norm asymptotic behavior. This inhomogeneous model arises from the motivation of breaking the scaling symmetries. Indeed, similar scaling-broken problems have been widely studied in the PDEs, for instance, \cite{TVZ2007} on the mass-energy double critical NLS, see also the recent paper \cite{Luo2023} and the references therein for more works. To establish the $\xi^3+\xi^5$ version profile decomposition, we first establish the corresponding inhomogeneous even curve $|\xi|^3+|\xi|^5$ version profile decomposition, which relies on the corresponding local convergence property and refined Strichartz estimate, and then apply the aforementioned divide-and-regroup argument to get the desired profile decomposition. For the $\xi^3+\xi^5$ version two-profile Strichartz norm asymptotic behavior as $|h_n\xi_n|\to \infty$, we mainly apply the homogenization result of Frank-Sabin \cite[Lemma B.1]{FS2018}. Finally, these profile decomposition and two-profile asymptotic consequences directly imply the precompactness Theorem \ref{T:Precompactness-3,5}.

The outline of this paper is as follows. In Section \ref{S:Homogeneous odd curve}, we study the homogeneous odd curve case, establish the associated profile decomposition Theorem \ref{T:Profile decomposition}, and then show the precompactness Theorem \ref{T:Precompactness}. In Section \ref{S:Inhomogeneous odd curve}, we turn to study the inhomogeneous odd curve case, establish the corresponding two-profile Strichartz norm asymptotic behavior Proposition \ref{P:Two-profile limit-3,5}, and then establish the associated profile decomposition Theorem \ref{T:Profile decomposition-3,5}.

\section{Homogeneous odd curve} \label{S:Homogeneous odd curve}
\subsection{Preliminaries}
We first recall the profile decomposition consequence of \cite[Proposition 1.5 \& Proposition 1.7]{DY2023} for the even curve $|\xi|^{\ell}$ case, and the corresponding classical $\ell=\{1,2,4\}$ even curve cases were extensively studied in \cite{BG1999,Bourgain1998,BV2007,CK2007,JPS2010,Keraani2001,MV1998}.
\begin{definition}[Orthogonal parameters: even curve] \label{D:Even orthogonality}
	We say the parameters $(h_n^j, x_n^j, \xi_n^j, t_n^j)$ are \textit{even pairwise orthogonal} if for arbitrary fixed $j\neq k$, these parameters satisfy one of the following:
	\begin{itemize}
		\item $(h_n^j,\xi_n^j) \neq (h_n^k,\xi_n^k)$ and
		\[\lim_{n\to \infty}\l(\frac{h_n^j}{h_n^k} + \frac{h_n^k}{h_n^j} + (h_n^j+h_n^k) |\xi_n^j-\xi_n^k| \r) = \infty;\]
		\item $(h_n^j,\xi_n^j) \equiv (h_n^k,\xi_n^k) \equiv (h_n, \xi_n)$ with $|h_n\xi_n| \to \infty$ and
		\[\lim_{n\to\infty} \l(\l|\frac{x_n^j-x_n^k}{h_n} - \frac{(t_n^j-t_n^k)\alpha (h_n \xi_n)^{\alpha-1}}{(h_n)^{\alpha}}\r| + \l|\frac{(t_n^j-t_n^k) (h_n \xi_n)^{\alpha-2}}{(h_n)^{\alpha}}\r| \r)= \infty;\]
		\item $(h_n^j,\xi_n^j) \equiv (h_n^k,\xi_n^k) \equiv (h_n, \xi_n)$ with $\xi_n \equiv 0$ and
		\[\lim_{n\to\infty} \l(\l|\frac{x_n^j-x_n^k}{h_n} \r| + \l|\frac{t_n^j-t_n^k}{(h_n)^{\alpha}}\r| \r)= \infty.\]
	\end{itemize}
\end{definition}

\begin{prop}[Profile decomposition: even curve \cite{DY2023}] \label{P:Profile decomposition-even}
	Let $(f_n)$ be a bounded sequence in $L^2(\bR)$. Then, up to subsequences, there exist a sequence of operators $([T_{*,n}^j])$ defined by
	\[[T_{*,n}^{j}]\phi(x):=[e^{it_n^j|\nabla|^{\ell}}]\l[(h_n^j)^{-\frac{1}{2}}e^{i(x-x_n^j)\xi_n^j}\phi\l(\frac{x-x_n^j}{h_n^j}\r)\r]\]
	with even pairwise orthogonal parameters $(h_n^j, x_n^j, \xi_n^j, t_n^j)$ and a sequence of functions $(\phi^j)\subset L^2(\bR)$ such that for every integer $J\in\bN$, we have the profile decomposition
	\begin{equation}\label{P:Profile decomposition-even-1}
		f_n=\sum_{j=1}^{J} [T_{*,n}^j]\phi^j+\omega_n^{J},
	\end{equation}
	where this decomposition possesses the following properties: firstly, the remainder term $\omega_n^{J}$ has vanishing Strichartz norm
	\begin{equation}\label{P:Profile decomposition-even-2}
		\lim_{J\to\infty}\limsup_{n\to\infty}\l\|[D^{\frac{\ell-2}{6}}][e^{it|\nabla|^{\ell}}]\omega_n^{J}\r\|_{L_{t,x}^6(\bR^{2})}=0;
	\end{equation}
	secondly, for each $j\leq J$, there holds the following $L_x^2$ weak convergence
	\begin{equation} \label{P:Profile decomposition-even-3}
		[T_{*,n}^j]^{-1} \omega_n^J \rightharpoonup 0 \quad \text{as} \quad n\to \infty.
	\end{equation}
\end{prop}

\begin{remark} \label{R:Profile decomposition-even-1}
	The even-orthogonality of parameters is essentially equivalent to the following fact: for arbitrary $j\neq k$, there holds
	\begin{equation}\label{P:Profile decomposition-even-4}
		[T_{*,n}^k]^{-1}[T_{*,n}^j]\rightharpoonup 0 \quad \text{as} \quad n\to \infty
	\end{equation}
	in the weak operator topology of $\mathcal{B}(L^2)$. Hence by \eqref{P:Profile decomposition-even-3}, for each $J\in\bN$, we have the $L^2$-limit orthogonality
	\begin{equation}\label{P:Profile decomposition-even-5}
		\lim_{n\to\infty}\l[\|f_n\|_{L^2(\bR)}^2-\l(\sum_{j=1}^{J}\|\phi^j\|_{L^2(\bR)}^2\r)-\|\omega_n^{J}\|_{L^2(\bR)}^2\r]=0.
	\end{equation}
	The even-orthogonality of parameters further implies that for each $j\neq k$ we have
	\begin{equation}\label{P:Profile decomposition-even-6}
		\lim_{n\to\infty}\l\|[D^{\frac{\ell-2}{6}}][e^{it|\nabla|^{\ell}}][T_{*,n}^j]\phi^j \cdot [D^{\frac{\ell-2}{6}}][e^{it|\nabla|^{\ell}}][T_{*,n}^k]\phi^k\r\|_{L_{t,x}^{3}(\bR^{2})}=0,
	\end{equation}
	which gives that for each $J\in\bN$ there holds the following Strichartz-limit orthogonality
	\begin{equation}\label{P:Profile decomposition-even-7}
		\limsup_{n\to\infty}\l(\l\|\sum_{j=1}^J[D^{\frac{\ell-2}{6}}][e^{it|\nabla|^{\ell}}][T_n^j]\phi^j\r\|_{L_{t,x}^6}^{6} -\sum_{j=1}^J\l\|[D^{\frac{\ell-2}{6}}][e^{it|\nabla|^{\ell}}][T_n^j]\phi^j\r\|_{L_{t,x}^6}^{6}\r)=0.
	\end{equation}
	For fixed $j\in \bN$, by adjusting the profiles, we can assume either $|h_n^j\xi_n^j| \to \infty$ or $\xi_n^j\equiv 0$.
\end{remark}

\begin{remark} \label{R:Profile decomposition-even}
	The operators $[T_{*,n}^{j}]$ in Proposition \ref{P:Profile decomposition-even} satisfy the following dislocation property: for fixed $j\neq k$, if in the weak operator topology there holds
	\[[T_{*,n}^k]^{-1}[T_{*,n}^j]\not\rightharpoonup 0 \quad \text{as} \quad n\to \infty,\]
	then, up to subsequences, there exists a unitary operator $[T^{j,k}] \in \mathcal{B}(L^2)$ such that in the strong operator topology
	\[[T_{*,n}^k]^{-1}[T_{*,n}^j] \to [T^{j,k}] \quad \text{as} \quad n\to \infty.\]
\end{remark}

\subsection{Profile decomposition}
Now we prove the desired homogeneous odd curve version profile decomposition Theorem \ref{T:Profile decomposition}. To begin with, we introduce the following Fourier projection operator
\[\widehat{[P_{E}]f}(\xi):= \widehat{f} \mathds{1}_{E}(\xi), \quad \widehat{[P_{+}]f}(\xi):= \widehat{f} \mathds{1}_{\bR_{+}}(\xi).\]
\begin{proof}[\textbf{Proof of Theorem \ref{T:Profile decomposition}}]
	For the sequence $(f_n)$, we first divide each function as
	\[f_n= [P_{+}] f_{n} + [P_{-}] f_{n}.\]
	Acting the projection operator $[P_{+}]$ on both sides of the aforementioned even curve version profile decomposition \eqref{P:Profile decomposition-even-1}, one can obtain the following decomposition
	\[[P_{+}]f_n = \sum_{k=1}^K [P_{+}] [T_{*,n}^{k}] \phi^k + [P_{+}] \omega_n^{K}.\]
	For each fixed $k$, by the expression
	\begin{align*}
		[P_{+}] [T_{*,n}^{k}] \phi^k(x) &= (h_n^k)^{-\frac{1}{2}} \int_{\bR_{+}} e^{i\l(\frac{x-x_n^k}{h_n^k}\r)\xi + i\frac{t_n^k}{(h_n^k)^{\ell}} \xi^{\ell}} \widehat{\phi}^k(\xi-h_n^k\xi_n^k) \ddd \xi,
	\end{align*}
	and noticing the fact that these parameters $(h_n^k,\xi_n^k)$ satisfy one of the following
	\[\xi_n^k \equiv 0, \quad \lim_{n\to\infty} h_n^k \xi_n^k \to +\infty, \quad \lim_{n\to\infty} h_n^k\xi_n^k \to -\infty,\]
	we can define the modified profile $\phi_{*,+}^k$ as
	\[\widehat{\phi}_{*,+}^k(\xi):= \begin{cases}
		\widehat{\phi}^k(\xi), & \xi_n^k \equiv 0 \\
		\widehat{\phi}^k(\xi), & \lim_{n\to\infty} h_n^k \xi_n^k \to +\infty \\
		0, & \lim_{n\to\infty} h_n^k\xi_n^k \to -\infty
	\end{cases}.\]
	Then for each fixed $k$ there holds
	\[\lim_{n\to\infty} \l\|[P_{+}] [T_{*,n}^{k}] \phi^k - [e^{it_n^k}\nabla^{\ell}] \l[(h_n^k)^{-\frac{1}{2}} e^{i(x-x_n^k)\xi_n^k} \phi_{*,+}^k \l(\frac{x-x_n^k}{h_n^k}\r)\r]\r\|_{L_x^2}=0.\]
	Therefore, for each $k$, we can modify the profiles and then put the error terms into the remainder to deduce the following decomposition
	\begin{align}
		[P_{+}]f_n &= \sum_{k=1}^K [e^{it_n^k}\nabla^{\ell}] \l[(h_n^k)^{-\frac{1}{2}} e^{i(x-x_n^k)\xi_n^k} \phi_{*,+}^k \l(\frac{x-x_n^k}{h_n^k}\r)\r] + [P_{+}] \omega_{n}^{K} + \omega_{n,*,+}^{K} \label{E:Profile decomposition-odd-1} \\
		&=: \sum_{k=1}^K [T_{*,n}^{k,+}] \phi_{*,+}^k + [P_{+}] \omega_{n}^{K} + \omega_{n,*,+}^{K}, \notag
	\end{align}
	where the remainder term $\omega_{n,*,+}^{K}$ satisfies
	\begin{equation} \label{E:Profile decomposition-odd-1.2}
		\lim_{n\to\infty} \|\omega_{n,*,+}^{K}\|_{L_x^2}=0.
	\end{equation}
	By the conclusion \eqref{P:Profile decomposition-even-2} and the Strichartz estimate \eqref{E:Odd Strichartz}, one can obtain that
	\[\lim_{J\to\infty}\limsup_{n\to\infty} \l\|[E_{\ell}] [P_{+}] \omega_n^{K} \r\|_{L_{t,x}^6} = \lim_{J\to\infty}\limsup_{n\to\infty} \l\|[P_{+}] [E_{\ell}] \omega_n^{K} \r\|_{L_{t,x}^6}=0, \quad \lim_{n\to\infty} \l\|[E_{\ell}] \omega_{n,*,+}^{K} \r\|_{L_{t,x}^6}=0.\]
	The parameters in $[T_{*,n}^{k,+}]$ satisfy the even pairwise orthogonality defined in Definition \ref{D:Even orthogonality} and by the definition of $\phi_{*,+}^k$, for fixed $k$, the parameters $(h_n^k,\xi_n^k)$ satisfy either $\xi_n^k\equiv 0$ or $\lim_{n\to\infty} h_n^k\xi_n^k = +\infty$. Hence we may assume $\xi_n^k\geq 0$ up to subsequences. Then as stated in Remark \ref{R:Profile decomposition-even-1}, for fixed $j\neq k$, the orthogonality of parameters further implies the following Strichartz-limit orthogonality
	\begin{equation} \label{E:Profile decomposition-odd-1.5}
		\lim_{n\to\infty} \l\| [E_{\ell}][T_{*,n}^{j,+}]\phi_{*,+}^j \cdot [E_{\ell}][T_{*,n}^{k,+}]\phi_{*,+}^k \r\|_{L_{t,x}^3} = 0.
	\end{equation}
	Meanwhile, for each $k\leq K$, the conclusions \eqref{P:Profile decomposition-even-3} and \eqref{E:Profile decomposition-odd-1.2} imply the following weak convergence
	\[[T_{*,n}^{k,+}]^{-1} \l([P_{+}]\omega_n^K + \omega_{n,*,+}^K\r) \rightharpoonup 0.\]
	In summary, we have shown that the decomposition \eqref{E:Profile decomposition-odd-1} is a profile decomposition for $[P_{+}]f_n$. Similarly, we can obtain the profile decomposition for $[P_{-}]f_n$ as follows
	\[[P_{-}]f_n=\sum_{k=1}^K [T_{*,n}^{k,-}] \phi_{*,-}^k + [P_{-}] \omega_{n}^{K} + \omega_{n,*,-}^{K},\]
	where the parameters are even pairwise orthogonal with $\xi_n^k \in \bR_{-}$ and, for each fixed $k$, the parameters $(h_n^k,\xi_n^k)$ satisfy either $\xi_n^k\equiv 0$ or $\lim_{n\to\infty} h_n^k\xi_n^k =-\infty$.
	Therefore, we obtain the following decomposition
	\begin{equation} \label{E:Profile decomposition-odd-2}
		f_n= \sum_{j=1}^J \l([T_{*,n}^{j,+}] \phi_{+}^j + [T_{*,n}^{j,-}] \phi_{-}^j\r) + \omega_{n,+}^J + \omega_{n,-}^J.
	\end{equation}
	Here the remainder terms satisfy the vanishing Strichartz norm
	\begin{equation} \label{E:Profile decomposition-odd-3}
		\lim_{J\to\infty}\limsup_{n\to\infty} \l\|[E_{\ell}] \l(\omega_{n,+}^{J} + \omega_{n,-}^{J}\r) \r\|_{L_{t,x}^6} = 0
	\end{equation}
	and, for each fixed $j\leq J$, the following weak convergence property
	\begin{equation} \label{E:Profile decomposition-odd-4}
		[T_{*,n}^{j,+}]^{-1} \omega_{n,+}^J \rightharpoonup 0, \quad [T_{*,n}^{j,-}]^{-1} \omega_{n,-}^J \rightharpoonup 0.
	\end{equation}
	Meanwhile, all the operators $[T_{*,n}^{j,+}]$ (as well as $[T_{*,n}^{j,-}]$) are even-orthogonal in the sense that their parameters are even pairwise orthogonal, and there holds the following $L^2$-limit orthogonality
	\begin{equation} \label{E:Profile decomposition-odd-4.2}
		\lim_{n\to\infty}\l[\|f_n\|_{L_x^2}^2- \l(\sum_{j=1}^{J} \l\|\phi_{+}^j \r\|_{L_x^2}^2 + \l\|\phi_{-}^j \r\|_{L_x^2}^2\r)-\|\omega_{n,+}^{J}\|_{L_x^2}^2-\|\omega_{n,-}^{J}\|_{L_x^2}^2\r]=0.
	\end{equation}

	Next, we will modify the decomposition \eqref{E:Profile decomposition-odd-2} to establish our desired profile decomposition \eqref{T:Profile decomposition-1}, then establish the odd pairwise orthogonality of parameters and show the orthogonality of profiles. We will achieve these by using some redefine-and-regroup arguments, which do not change the properties \eqref{E:Profile decomposition-odd-3} and \eqref{E:Profile decomposition-odd-4} and \eqref{E:Profile decomposition-odd-4.2}.
	
	To begin with, we introduce the following frequency-reflected notation
	\[[\overline{T_{*,n}^{j,+}}] \phi(x):= [e^{it_n^j}\nabla^{\ell}] \l[(h_n^j)^{-\frac{1}{2}} e^{-i(x-x_n^j)\xi_n^j} \phi \l(\frac{x-x_n^j}{h_n^j}\r)\r].\]
	For the fixed profile $[T_{*,n}^{1,+}]\phi_{+}^1$ in \eqref{E:Profile decomposition-odd-2}, one can see that all the profiles $[T_{*,n}^{j,+}]\phi_{+}^j$ with $j \neq 1$ are odd-orthogonal to this profile since all the corresponding frequency parameters are nonnegative. Then we are going to show two facts: first, there is at most one term $[T_{*,n}^{j,-}]\phi_{-}^j$ which are not odd-orthogonal to $[T_{*,n}^{1,+}]\phi_{+}^1$, and this term satisfy the following operator weak convergence
	\[[T_{*,n}^{j,-}]^{-1}[\overline{T_{*,n}^{1,+}}] \not\rightharpoonup 0;\]
	second, for the aforementioned term $[T_{*,n}^{j,-}]\phi_{-}^j$, we can find a functions $\phi_{*,-}^j$ such that
	\[\lim_{n\to\infty} \l\|[T_{*,n}^{j,-}]\phi_{-}^j - [\overline{T_{*,n}^{1,+}}] \phi_{*,-}^j\r\|_{L_x^2} = 0.\]
	Indeed, as shown below, these two facts hold for each fixed profile $[T_{*,n}^{j,\pm}]\phi_{\pm}^j$.
	
	For the first fact, we prove by contradiction. Recalling that either $h_n^1\xi_n^1\to +\infty$ or $\xi_n^1 \equiv0$, we first investigate the case $h_n^1\xi_n^1\to +\infty$. Assume that there exist two parameters $j$ and $k$ such that
	\[[T_{*,n}^{j,-}]^{-1}[\overline{T_{*,n}^{1,+}}] \not\rightharpoonup 0, \quad [T_{*,n}^{k,-}]^{-1}[\overline{T_{*,n}^{1,+}}] \not\rightharpoonup 0.\]
	Then by the dislocation property in Remark \ref{R:Profile decomposition-even}, up to subsequences, we know there exist two unitary operators $[T^{1,j}]$ and $[T^{1,k}]$ such that
	\[[T_{*,n}^{j,-}]^{-1}[\overline{T_{*,n}^{1,+}}] \to [T^{1,j}], \quad [T_{*,n}^{k,-}]^{-1}[\overline{T_{*,n}^{1,+}}] \to [T^{1,k}].\]
	This relation further implies
	\[[T_{*,n}^{j,-}]^{-1} [T_{*,n}^{k,-}] = [T_{*,n}^{j,-}]^{-1} [T_{*,n}^{1,+}] [T_{*,n}^{1,+}]^{-1} [T_{*,n}^{k,-}] \to [T^{1,j}] [T^{1,k}]^{-1},\]
	which leads to a contradiction to the operator-orthogonality conclusion \eqref{P:Profile decomposition-even-4}. Hence, for the case $h_n^1\xi_n^1 \to +\infty$, we have shown that there is at most one term $[T_{*,n}^{j,-}]\phi_{-}^j$ such that
	\[[T_{*,n}^{j,-}]^{-1}[\overline{T_{*,n}^{1,+}}] \not\rightharpoonup 0.\]
	By applying similar arguments, one can prove the case $\xi_n^1 \equiv0$. Therefore, we complete the proof of the first fact. Then we prove the second fact. For the aforementioned term $[T_{*,n}^{j,-}]\phi_{-}^j$ which is not odd-orthogonal to $[T_{*,n}^{1,+}]\phi_{+}^1$, we see that
	\begin{equation} \label{E:Profile decomposition-odd-5}
		[T_{*,n}^{j,-}]^{-1}[\overline{T_{*,n}^{1,+}}] \to [T^{1,j}], \quad \Rightarrow \quad [T^{1,j}]^{-1} [T_{*,n}^{j,-}]^{-1}[\overline{T_{*,n}^{1,+}}] \to [E],
	\end{equation}
	where $[E]$ is the identity operator. Notice that
	\begin{equation} \label{E:Profile decomposition-odd-6}
		[T_{*,n}^{j,-}]\phi_{-}^j= [T_{*,n}^{j,-}][T^{1,j}][T^{1,j}]^{-1}\phi_{-}^j = [T_{*,n}^{j,-}][T^{1,j}] \phi_{*,-}^j, \quad \phi_{-}^{j,*}:= [T^{1,j}]^{-1} \phi_{-}^j.
	\end{equation}
	These two relations \eqref{E:Profile decomposition-odd-5} and \eqref{E:Profile decomposition-odd-6} further imply that
	\[\lim_{n\to\infty} \l\|[T_{*,n}^{j,-}]\phi_{-}^j - [\overline{T_{*,n}^{1,+}}]\phi_{-}^{j,*} \r\|_{L_x^2}=0.\]
	This completes the proof of the second fact.
	
	Now we establish the desired profile decomposition by regrouping the profiles in \eqref{E:Profile decomposition-odd-2}. For each integer $j\in\bN$ and the term $[T_{*,n}^{j,\pm}]\phi_{\pm}^j$, we may assume $\pm$ is $+$ without loss of generality. We find all the integers $k_j$ such that $k_j$ are not odd-orthogonal to $j$, and then group these profiles together as a new profile
    \begin{equation} \label{E:Profile decomposition-odd-6.5}
        g_n^j:= [T_{*,n}^{j,+}]\phi_{+}^j + \sum_{k_j: k_j \not\perp j} [T_{*,n}^{k_j,-}]\phi_{-}^{k_j}.
    \end{equation}
	The aforementioned two facts imply that there exists only one $k_j$ and we can further redefine the following profile
	\[\widetilde{g}_n^j:= [T_{*,n}^{j,+}] \phi_{+}^j + [\overline{T_{*,n}^{j,+}}] \phi_{-}^{k_{j},*},\]
	which satisfies
	\begin{equation} \label{E:Profile decomposition-odd-7}
		\lim_{n\to\infty} \l\|e_n^j\r\|_{L_x^2} = 0, \quad e_n^j:= \widetilde{g}_n^j - g_n^j.
	\end{equation}
	Then putting the error term $e_n^j$ into the remainder and denoting
	\[\phi_{+}^j:= \phi_{+}^j, \quad \phi_{-}^j:= \phi_{-}^{k_{j},*}, \quad [T_n^j](\phi_{+}^j,\phi_{-}^j):= [T_{*,n}^{j,+}]\phi_{+}^j + [\overline{T_{*,n}^{j,+}}]\phi_{-}^j,\]
	we obtain the desired profile decomposition \eqref{T:Profile decomposition-1} as follows
	\[f_n = \sum_{j=1}^J [T_n^j](\phi_{+}^j,\phi_{-}^j) + \omega_{n,+}^J +\omega_{n,-}^J.\]
	
	Next, we show the desired properties of this profile decomposition. Firstly, the estimates \eqref{E:Profile decomposition-odd-3} and \eqref{E:Profile decomposition-odd-7} together with \eqref{E:Odd Strichartz} imply the vanishing Strichartz norm property \eqref{T:Profile decomposition-2} for the remainder term. Secondly, for each $j\leq J$, the operators strong convergence relation \eqref{E:Profile decomposition-odd-5} and the functions weak convergence \eqref{E:Profile decomposition-odd-4} imply the following weak convergence conclusion
	\[[\overline{T_{*,n}^{j,+}}]^{-1} \omega_{n,-}^J \rightharpoonup0,\]
	which, by \eqref{E:Profile decomposition-odd-4}, further gives the desired $L^2$ weak convergence conclusion \eqref{T:Profile decomposition-3}. Thirdly, by the construction of $g_n^j$ in \eqref{E:Profile decomposition-odd-6.5}, we know that for each fixed $j\neq k$ there holds the following $\mathcal{B}(L^2)$ weak operator topology convergence
	\[[T_{*,n}^{j,+}]^{-1} [T_{*,n}^{k,+}] \rightharpoonup 0, \quad [T_{*,n}^{j,+}]^{-1} [\overline{T_{*,n}^{k,+}}] \rightharpoonup 0.\]
	These correspond to the odd-orthogonality of parameters defined in Definition \ref{D:Odd orthogonality}, and directly give the desired $L^2$-limit orthogonality \eqref{T:Profile decomposition-5}. Finally, by the even Strichartz-limit orthogonality \eqref{E:Profile decomposition-odd-1.5} and the following identity
	\[\l\|[E_{\ell}][T_{*,n}^{k,+}]\phi_+^k [E_{\ell}][\overline{T_{*,n}^{j,+}}]\phi_{-}^j\r\|_{L_{t,x}^3}= \l\|[E_{\ell}][T_{*,n}^{k,+}]\phi_+^k [E_{\ell}][T_{*,n}^{j,+}]\overline{\phi}_-^j\r\|_{L_{t,x}^3},\]
	we obtain the desired odd Strichartz-limit orthogonality \eqref{T:Profile decomposition-6} and complete the proof.
\end{proof}

\subsection{Characterization of precompactness}
In this subsection, we prove the precompactness result Theorem \ref{T:Precompactness}. The arguments are standard, and our proof will be sketchy. Further details can be seen in, for instance, \cite[Proof of Theorem 1.8]{JPS2010} and \cite[Proof of Theorem A]{DY2023}. 
\begin{proof}[\textbf{Proof of Theorem \ref{T:Precompactness}}]
	Using the profile decomposition Theorem \ref{T:Profile decomposition}, one can follow some classical arguments to obtain the following conclusion: there exist two functions $\phi_{\pm}^{j_0}$ in $L^2$ such that the extremal sequence $f_n= [T_n^{j_0}] (\phi_{+}^{j_0}, \phi_{-}^{j_0})$ and
	\begin{equation*}
		\lim_{n\to\infty} \l\|[E_{\ell}][T_n^{j_0}] (\phi_{+}^{j_0}, \phi_{-}^{j_0}) \r\|_{L_{t,x}^{6}} =\lim_{n\to\infty} \l\|[E_{\ell}]\l(e^{i(\cdot) h_n^{j_0} \xi_n^{j_0}} \phi_{+}^{j_0} + e^{-i(\cdot) h_n^{j_0} \xi_n^{j_0}} \phi_{-}^{j_0}\r) \r\|_{L_{t,x}^{6}} =\mathbf{M}_{\ell},
	\end{equation*}
	where the parameters $(h_n^{j_0}, x_n^{j_0}, \xi_n^{j_0}, t_n^{j_0}) \subset \bR_{+}\times\bR^3$. The readers are referred to \cite[proof of Theorem A]{DY2023} to see the proof of this conclusion.
	
	Recall that the parameters $(h_n^{j_0},\xi_n^{j_0})$ satisfy either $\xi_n^{j_0} \equiv0$, $h_n^{j_0} \xi_n^{j_0} \to +\infty$ or $h_n^{j_0} \xi_n^{j_0} \to -\infty$. By exchanging $\phi_{+}^{j}$ and $\phi_{-}^{j}$, we can assume $\xi_n^{j_0}\geq 0$ without loss of generality. Next, we divide the proof into two cases.
	
	If $\xi_n^{j_0}\equiv 0$, then up to symmetries the sequence $(f_n)$ strongly converges to $\phi_{+}^{j_0} + \phi_{-}^{j_0}$ in $L^2$, and this is an extremal function for $\mathbf{M}_{\ell}$. Otherwise if $h_n^{j_0}\xi_n^{j_0} \to +\infty$, then by the two-profile Strichartz norm asymptotic Proposition \ref{P:Two-profile limit-odd}, we conclude that
	\begin{align*}
		\mathbf{M}_{\ell} &\leq \l[\ell(\ell-1)/5\r]^{-1/6} \l(\l\|[e^{it\Delta}]\phi_{+}^{j_0}\r\|_{L_{t,x}^6}^2 + \l\|[e^{it\Delta}]\phi_{-}^{j_0}\r\|_{L_{t,x}^6}^2 \r)^{1/2} \\
		& \leq \l[\ell(\ell-1)/5\r]^{-1/6} \mathbf{M}_2 \l(\|\phi_{+}^{j_0}\|_{L_x^2}^2 + \|\phi_{-}^{j_0}\|_{L_x^2}^2\r)^{1/2} \\
		&=\l[\ell(\ell-1)/5\r]^{-1/6} \mathbf{M}_2.
	\end{align*}
	Here in the last identity, we have used the $L^2$-limit orthogonality \eqref{T:Profile decomposition-5}.
	
	On the other hand, by taking the following sequence
	\[f_*(x)=e^{-|x|^2}, \quad \widetilde{f}_N^{*}(x):= e^{ix N} f_*(x) + e^{-ix N} \overline{f}_*(x), \quad \widetilde{f}_N:= \frac{\widetilde{f}_N^{*}}{\|\widetilde{f}_N^{*}\|_{L_x^2}}\]
	and recalling the fact that Gaussians are the extremals for $\mathbf{M}_{2}$ shown in \cite{Foschi2007,HZ2006}, we can use the Proposition \ref{P:Two-profile limit-odd} to deduce that
	\[\mathbf{M}_{\ell} \geq  \lim_{N\to\infty} \l\|[E_{\ell}] \widetilde{f}_N\r\|_{L_{t,x}^6} =\l[\ell(\ell-1)/5\r]^{-1/6} \mathbf{M}_2.\]
	Therefore, if $\mathbf{M}_{\ell} > \l[\ell(\ell-1)/5\r]^{-1/6} \mathbf{M}_2$, then the aforementioned $h_n^{j_0}\xi_n^{j_0} \to +\infty$ cannot happen, and all the extremal sequences are precompact up to symmetries. If $\mathbf{M}_{\ell} = \l[\ell(\ell-1)/5\r]^{-1/6} \mathbf{M}_2$, then the aforementioned normalized sequence $(\widetilde{f}_N)$ gives an extremal sequence that is not precompact up to symmetries. Thus the proof is finished.
\end{proof}

\section{Inhomogeneous odd curve} \label{S:Inhomogeneous odd curve}
\subsection{Approximate symmetries} \label{SS:Approximate symmetries}
For the $\xi^3 +\xi^5$ case, this subsection is devoted to studying the effect of frequency parameters $\xi_n$ with $h_n\xi_n \to \pm \infty$. First, we introduce the inhomogeneous odd approximate operator
\begin{equation} \label{E:Approximate operator-3,5}
	[\widetilde{T}_{5,N}^{\pm}]f(t,x):= \int_{\bR} \l| N^{-3} \l(\xi\pm N + (\xi \pm N)^3\r) \r|^{1/6} e^{ix\xi + it \widetilde{\varphi}_{5,N}^{\pm}(\xi)} \widehat{f}(\xi) \ddd\xi
\end{equation}
and the associated limit operator
\[[e^{\pm10it\Delta}]f(x):= \int_{\bR} e^{ix\xi \pm 10 it \xi^2} \widehat{f}(\xi) \ddd \xi,\]
where the phase function $\widetilde{\varphi}_{5,N}^{\pm}(\xi)$ is given by
\[\widetilde{\varphi}_{5,N}^{\pm}(\xi) := N^{-3}\l(\xi^5 \pm 5N\xi^4 +10N^2\xi^3 \pm 10N^3\xi^2 + \xi^3 \pm 3N\xi^2\r).\]
Notice that there pointwisely holds the following relation
\[\lim_{N\to\infty} \widetilde{\varphi}_{5,N}^{\pm}(\xi) = \pm 10 \xi^2.\]
Hence, similar to the asymptotic \schrodinger behavior studied in \cite[Section 4]{DY2023} and \cite[Section 4]{FS2018}, one can use the stationary phase method to obtain the following lemma.
\begin{lemma}[Asymptotic Schr\"odinger: inhomogeneous curve] \label{L:Asymptotic Schrodinger-3,5}
	For any $f\in L^2(\bR)$, there holds
	\[\lim_{N\to +\infty} \l\|[\widetilde{T}_{5,N}^{\pm}]f(t,x) - [e^{\pm10it\Delta}] f(x) \r\|_{L_{t,x}^{6}} =0.\]
\end{lemma}

Next, we use this inhomogeneous asymptotic \schrodinger lemma to investigate the asymptotic behavior of the reduced two-profile Strichartz norms. The homogeneous $\xi^3$ case was first studied by Frank and Sabin \cite[Lemma 2.4]{FS2018}, where they have applied a homogenization result to compute this specific value. Their consequence can be parallelly extended to general homogeneous odd curve cases, here we show that this can be further extended to general inhomogeneous odd curve cases.
\begin{prop}[Two-profile Strichartz norm asymptotic] \label{P:Two-profile limit-3,5}
	For arbitrary two functions $g_1$ and $g_2$ in $L^2(\bR)$, denoting
	\[\widetilde{g}_N(x):= e^{ix N} g_1(x) + e^{-ix N} g_2(x),\]
	then there holds
	\begin{equation} \label{P:Two-profile limit-3,5-1}
		\lim_{N\to\infty} \l\|[\widetilde{E}_{5}]\widetilde{g}_N \r\|_{L_{t,x}^6} \leq 4^{-1/6} \l(\l\|[e^{it\Delta}]g_1\r\|_{L_{t,x}^6}^2 + \l\|[e^{it\Delta}]g_2\r\|_{L_{t,x}^6}^2 \r)^{1/2},
	\end{equation}
	and the equality holds if and only if
	\[\l|[e^{it\Delta}]g_1(x)\r|= \l|[e^{-it\Delta}]g_2(x)\r| ,\quad \text{a.e.} \;\; (t,x) \in \bR^2.\]
\end{prop}

\begin{remark}[Conjugate-profile] \label{R:Two-profile limit-3,5}
	By this two-profile Strichartz norm asymptotic result, we directly obtain the following conclusion: for any function $f\in L^2(\bR)$, denoting
	\[\widetilde{f}_N(x):= e^{ix N} f(x) + e^{-ix N} \overline{f}(x),\]
	then there holds
	\[\lim_{N\to\infty} \l\|[\widetilde{E}_{5}]\widetilde{f}_N \r\|_{L_{t,x}^6} = 2^{1/6} \l\|[e^{it\Delta}]f\r\|_{L_{t,x}^6}.\]
\end{remark}

\begin{proof}[\textbf{Proof of Proposition \ref{P:Two-profile limit-3,5}}]
	By changing of variables, there holds{\footnotesize
	\begin{align*}
		\l\|[\widetilde{E}_{5}] \widetilde{f}_N \r\|_{L_{t,x}^6} &= \l\|e^{ix N +it(N^2+1)} [\widetilde{T}_{5,N}^{+}]g_1(t,x+3N^{-1}t+5N t) + e^{-ix N -it(N^2+1)} [\widetilde{T}_{5,N}^{-}] g_2(t,x+3N^{-1}t+5N t) \r\|_{L_{t,x}^6} \\
		&=\l\|e^{ix N +it(N^2+1)} [\widetilde{T}_{5,N}^{+}]g_1 + e^{-ix N -it(N^2+1)} [\widetilde{T}_{5,N}^{-}] g_2 \r\|_{L_{t,x}^6}.
	\end{align*}}
	Then applying the inhomogeneous asymptotic \schrodinger Lemma \ref{L:Asymptotic Schrodinger-3,5} and noticing that the function $e^{i\theta}$ with $\theta\in\bR$ is uniformly bounded, we can further deduce
	\begin{align*}
		\l\|[\widetilde{E}_{5}]\widetilde{f}_N \r\|_{L_{t,x}^6} &= \l\|e^{ixN +it(N^2+1)} [e^{10it\Delta}] g_1 + e^{-ixN -it(N^2+1)} [e^{-10it\Delta}] g_2 \r\|_{L_{t,x}^6} + o_{N\to\infty}(1).
	\end{align*}
	 Next, we introduce the following function
	\[\psi_{t,x}(\theta):= \l|e^{i(\theta_1+\theta_2)}[e^{10it\Delta}] g_1(x) + e^{-i(\theta_1+\theta_2)}[e^{-10it\Delta}]
	 g_2(x)\r|^6, \quad \theta \in [\bR/(2\pi\bZ)]^2.\]
	Then for almost everywhere $(t,x)\in\bR^2$, the function $\psi_{t,x}(\theta)$ is continuous and its maximum value satisfies
	\[\int_{\bR^2} \max_{\theta} \psi_{t,x}(\theta) \ddd x\ddd t \lesssim \int_{\bR^2} \l|[e^{10it\Delta}] g_1(x)\r|^6 + \l|[e^{-10it\Delta}]
	g_2(x)\r|^6 \ddd x \ddd t <\infty.\]
	Hence, we can apply the homogenization result of Frank-Sabin \cite[Lemma B.1]{FS2018} to obtain
	\begin{align*}
		\l\|[\widetilde{E}_{5}] \widetilde{f}_N \r\|_{L_{t,x}^6}^6 &= \int_{\bR} \frac{1}{2\pi} \int_0^{2\pi} \int_{\bR} \l|e^{i\theta} [e^{10it\Delta}]g_1(x) + e^{-i\theta} [e^{-10it\Delta}] g_2(x) \r|^6 \ddd x \ddd \theta \ddd t + o_{N\to\infty}(1).
	\end{align*}
	Notice the inequality
	\[\frac{1}{2\pi} \int_0^{2\pi} |e^{i\theta} z_1 +e^{-i\theta} z_2|^r \ddd \theta \leq (|z_1|^2+|z_2|^2)^{r/2} \int_0^{2\pi} (1+\cos\theta)^{r/2} \frac{\ddd \theta}{2\pi}\]
	with the equality holds if and only if $2|z_1||z_2|= |z_1|^2 +|z_2|^2$. This inequality can be found in \cite[Equation (2.4)]{FS2018}. We then apply the triangle inequality to deduce
	\begin{align*}
		\l\|[\widetilde{E}_{5}] \widetilde{f}_N \r\|_{L_{t,x}^6}^6 &\leq \frac{5}{2} \int_{\bR^2} \l(\l|[e^{10it\Delta}]g_1(x)\r|^2 + \l|[e^{-10it\Delta}] g_2(x)\r|^2\r)^3 \ddd x \ddd t + o_{N\to\infty}(1) \\
		&\leq \frac{5}{2} \l(\l\|[e^{10it\Delta}]g_1(x)\r\|_{L_{t,x}^6}^2 + \l\|[e^{10it\Delta}]g_2(x)\r\|_{L_{t,x}^6}^2\r)^{3} + o_{N\to\infty}(1) \\
		&= \frac{1}{4} \l(\l\|[e^{it\Delta}]g_1(x)\r\|_{L_{t,x}^6}^2 + \l\|[e^{it\Delta}]g_2(x)\r\|_{L_{t,x}^6}^2\r)^{3} + o_{N\to\infty}(1)
	\end{align*}
	Here the first inequalities can be equal if and only if
	\[2\l|[e^{10it\Delta}]g_1(x)\r| \cdot \l|[e^{-10it\Delta}] g_2(x)\r| = \l|[e^{10it\Delta}]g_1(x)\r|^2 + \l|[e^{-10it\Delta}] g_2(x)\r|^2 \quad  \text{a.e.} \;\; (t,x) \in \bR^2,\]
	and the second inequalities can be equal if and only if there exists some $k_1\in\bR_{+}$ such that
	\[\l|[e^{10it\Delta}]g_1(x)\r| = k_1 \l|[e^{-10it\Delta}] g_2(x)\r| \quad \text{a.e.} \;\; (t,x) \in \bR^2.\]
	In summary, these two inequalities can be equal simultaneously if and only if
	\[\l|[e^{it\Delta}]g_1(x)\r|= \l|[e^{-it\Delta}]g_2(x)\r| ,\quad \text{a.e.} \;\; (t,x) \in \bR^2.\]
	Hence, we obtain the desired Proposition \ref{P:Two-profile limit-3,5}. Then, noticing the identity
	\[[e^{-10it\Delta}] \overline{f}(x) = \overline{[e^{10it\Delta}]f(x)},\]
	we directly conclude the Remark \ref{R:Two-profile limit-3,5}.
\end{proof}

Next, for the approximate operator $[\widetilde{T}_{5,N}^{\pm}]$, we investigate its local convergence property.
\begin{lemma}\label{L:Approximate local smooth-3,5}
	Let $\varphi$ be a Schwartz function and $\eta\geq 100$. Then for arbitrary function $f\in L^2(\bR)$ whose Fourier support in $\{\xi\in \bR: \xi\geq -\eta\}$, we have
	\[\int_{\bR^2} \varphi(x) \Big|[\psi_{\eta}(D)] [\widetilde{T}_{5}^{\eta}]f(t,x)\Big|^2 \ddd x\ddd t \lesssim \|f\|_{L_x^2}^2,\]
	where the operator $[\psi_{\eta}(D)][\widetilde{T}_{5}^{\eta}]$ is defined as
	\[[\psi_{\eta}(D)][\widetilde{T}_{5}^{\eta}]f(t,x):= \frac{1}{2\pi} \int_{\bR} \psi_{\eta}(\xi) \l| \l(\xi+ \eta + (\xi + \eta)^3\r)/\eta^3 \r|^{1/6} e^{ix\xi +it \phi_{\eta}(\xi)} \widehat{f}(\xi) \ddd \xi\]
	with
	\[\psi_{\eta}(\xi):=|\eta|^{1/2}|\xi|^{1/2}\l|\xi+\eta\r|^{-\frac{1}{2}}, \quad \phi_{\eta}(\xi):= \eta^{-3}\l(\xi^5 + 5\eta \xi^4 +10\eta^2\xi^3 + 10\eta^3\xi^2 + \xi^3 + 3\eta\xi^2\r).\]
\end{lemma}

\begin{proof}[\textbf{Proof of Lemma \ref{L:Approximate local smooth-3,5}}]
	By applying the Plancherel theorem and Schur's test, we only need to prove
	\begin{equation} \label{E:Approximate local smooth-1}
		\sup_{\xi\geq -\eta} \int_{\bR} \widehat{\phi}(\widetilde{\xi}-\xi) |\xi|^{1/2} |\widetilde{\xi}|^{1/2} \delta \l[\phi_{\eta}(\xi)- \phi_{\eta}(\widetilde{\xi})\r] \ddd \widetilde{\xi} <\infty
	\end{equation}
	independently of $\eta$. For every $(\eta,\xi)\in \bR^2$ with $\xi\geq -\eta$, define the function
	\[h_{\eta}(\xi):= \xi^5 + 5\eta \xi^4 +10\eta^2\xi^3 + 10\eta^3\xi^2 + \xi^3 + 3\eta\xi^2\]
	and the set
	\[\mathcal{V}_{\xi,\eta}:=\l\{\widetilde{\xi}: \widetilde{\xi} \geq -\eta, h_{\eta}(\widetilde{\xi}) = h_{\eta}(\xi) \r\}.\]
	Notice that there are at most five points in $\mathcal{V}_{\xi,\eta}$ for each fixed $(\eta,\xi)$. Hence, a direct computation shows that
	\begin{align*}
		&\sup_{\xi\geq -\eta} \int_{\bR} \widehat{\phi}(\widetilde{\xi}-\xi) |\xi|^{1/2} |\widetilde{\xi}|^{1/2} \delta \l[\phi_{\eta}(\xi)- \phi_{\eta}(\widetilde{\xi})\r] \ddd \widetilde{\xi} \\
		&\leq \sup_{\xi\geq -\eta} \max_{\widetilde{\xi}\in \mathcal{V}_{\xi,\eta}} \frac{\|\widehat{\phi}\|_{L^{\infty}} |\xi|^{1/2}|\widetilde{\xi}|^{1/2} |\eta|^{3}}{\Big| \partial_{\widetilde{\xi}} h_{\eta}(\widetilde{\xi}) \Big|} \\
		&= \sup_{\xi\geq -\eta} \max_{\widetilde{\xi}\in \mathcal{V}_{\xi,\eta}} \frac{\|\widehat{\phi}\|_{L^{\infty}} |\xi|^{1/2}|\widetilde{\xi}|^{1/2} |\eta|^3}{5|\widetilde{\xi}| \l| \widetilde{\xi}^3+4\eta\widetilde{\xi}^2 +6 \eta^2\widetilde{\xi} +4\eta^3 +\frac{3\widetilde{\xi}}{5} +\frac{6\eta}{5} \r|}.
	\end{align*}
	Next, denoting the function
	\[\phi^{*}(\widetilde{\xi},\eta) := \widetilde{\xi}^3+4\eta\widetilde{\xi}^2 +6 \eta^2\widetilde{\xi} +4\eta^3 +\frac{3\widetilde{\xi}}{5} +\frac{6\eta}{5},\]
	we are going to show the following two estimates
	\begin{equation} \label{E:Approximate local smooth-2}
		|\phi^{*} (\widetilde{\xi},\eta)| \gtrsim |\eta|^{3}, \quad |\xi| \sim |\widetilde{\xi}|
	\end{equation}
	for all $\widetilde{\xi}\in \mathcal{V}_{\xi,\eta}$ with $\xi\geq -\eta$. These two estimates can directly give the desired upper bound for \eqref{E:Approximate local smooth-1} and thus complete the proof.
	
	Firstly, we will show that for all $\widetilde{\xi}\in \mathcal{V}_{\xi,\eta}$ with $\xi\geq -\eta$ there holds
	\[\l|\phi^{*} (\widetilde{\xi},\eta)/\eta^{3}\r| =\l|\l(\widetilde{\xi}^3+4\eta\widetilde{\xi}^2 +6 \eta^2\widetilde{\xi} +4\eta^3 +\frac{3\widetilde{\xi}}{5} +\frac{6\eta}{5}\r)/\eta^3\r| \gtrsim 1.\]
	Denote $\widetilde{y}:=\widetilde{\xi}/\eta$ and recall the Fourier support assumption $\widetilde{\xi}\geq -\eta$. We only need to investigate the case $\widetilde{y} \in [-1,0)$ since the case $\widetilde{y}\geq 0$ is obvious. Then by the assumption $\eta\geq 100$ we conclude
	\[\l|\widetilde{y}^3 + 4\widetilde{y}^2 + 6\widetilde{y} +4+ \frac{3\widetilde{y}}{5 \eta^2} + \frac{6}{5 \eta^2}\r| \gtrsim 1.\]
	
	Secondly, we will show that $|\xi/\widetilde{\xi}| \gg 1$ is impossible, which implies our desired result $|\widetilde{\xi}| \sim |\xi|$ by symmetry. We further denote $y:=\xi/\eta$, then the set $\mathcal{V}_{\xi,\eta}$ can be rewritten as
	\[\mathcal{V}_{\xi,\eta}= \l\{\widetilde{y}\geq -1: \widetilde{y} \neq 0,  h_{\eta}^{*}(y) - h_{\eta}^{*}(\widetilde{y}) =0\r\},\]
	where
	\begin{equation*}
		h_{\eta}^{*}(y):= y^5 +5y^4 +10 y^3 + 10y^2 +\frac{y^3}{\eta^2} + \frac{3y^2}{\eta^4}.
	\end{equation*}
	Thus, $|y/\widetilde{y}| \gg 1$ with $y>0$ will directly lead to a contradiction. Next, we consider the case $|y/\widetilde{y}|\gg1$ with $-1\leq y<0$. Indeed, due to these ranges of the parameters $(y,\widetilde{y}, \eta)$, there hold
	\[|h_{\eta}^{*}(y)| \geq y^2, \quad |h_{\eta}^{*}(y)| \leq 15y^2.\]
	Hence, the case $|y/\widetilde{y}|\gg1$ with $-1\leq y<0$ also leads to a contradiction. In summary, we obtain the desired conclusion that $|\xi/\widetilde{\xi}|\gg 1$ is impossible, and thus finish the proof.
\end{proof}

\begin{remark}\label{R:Approximate local smooth-3,5}
	Let $\varphi$ be a Schwartz function and $\eta\leq -100$. Then for arbitrary function $f\in L^2(\bR)$ whose Fourier support in $\{\xi \in \bR: \xi\leq -\eta\}$, we have
	\[\int_{\bR^2} \varphi(x) \Big|[\psi_{\eta}(D)] [\widetilde{T}_{5}^{\eta}]f(t,x)\Big|^2 \ddd x\ddd t\lesssim \|f\|_{L_x^2}^2.\]
\end{remark}

Based on the aforementioned local smoothing consequences Lemma \ref{L:Approximate local smooth-3,5} and Remark \ref{R:Approximate local smooth-3,5}, as well as the identity $[\widetilde{T}_{5}^{\eta}]f(t,x) =[\widetilde{T}_{5}^{-\eta}] \overline{f}(-t,-x)$, one can do exactly the same arguments as in \cite[Lemma 3.2 \& Lemma 4.4]{DY2022b} to deduce the following local convergence lemma for approximate operators. To avoid too much repetition, we omit the proof.
\begin{lemma}[Approximate local convergence \& pointwise convergence] \label{L:Approximate local convergence}
	For frequency parameters $(\xi_n) \subset \bR$ and a sequence of functions $(f_n) \subset L^2(\bR)$, if
	\[(\xi_n)\subset \bR_{+}, \quad \xi_n\to +\infty, \quad \supp \widehat{f}_n \subset \{\xi\in \bR: \xi \geq -\xi_n\}, \quad f_n\rhu 0\]
	or
	\[(\xi_n)\subset \bR_{-}, \quad \xi_n\to -\infty, \quad \supp \widehat{f}_n \subset \{\xi\in \bR: \xi \leq -\xi_n\}, \quad f_n\rhu 0,\]
	then up to subsequences we have
	\[[\widetilde{T}_{5}^n]f_n\to0\]
	strongly in $L_{loc, t,x}^2(\bR^2)$ and hence $[\widetilde{T}_{5}^n]f_n(t,x)\to 0$ almost everywhere in $\bR^2$. Here the approximate operator $[\widetilde{T}_{5}^n]$ is defined as in \eqref{E:Approximate operator-3,5} with $\xi_n$ substituting $N$.
\end{lemma}

\subsection{Refined Strichartz} \label{SS:Refined Strichartz}
Besides the approximate local convergence Lemma \ref{L:Approximate local convergence} above, another main building block for the profile decomposition Theorem \ref{T:Profile decomposition-3,5} is a refinement of Strichartz estimate, see Proposition \ref{P:Refined Strichartz} below. To begin with, we introduce some dyadic analysis terminologies and basic properties.
\begin{definition}[Dyadic intervals]
	An interval $\tau\subset \bR$ is \textit{dyadic} if it can be written as $\tau=[k,k+1)2^{j}$ with $k\in \bZ$ and $j \in \bZ$. For any interval $\tau$, we denote by $|\tau|$ its length, and $c(\tau)$ or $c_{\tau}$ its center.
\end{definition}
\begin{definition}[Adjacent]
	Two intervals are said to be \textit{adjacent} if they have the same length and share an extremity. For two dyadic intervals $\tau$ and $\tau'$, we write $\tau\sim \tau'$ if $|\tau|=|\tau'|$ and they are not adjacent, their parents are not adjacent, their $2$-parents (grandparents) are not adjacent, ..., their $(N_{\ell}-1)$-parents are not adjacent, but their $N_{\ell}$-parents are adjacent. Here the number $N_{\ell} \in \bN$ will be determined later in \eqref{E:Adjacent number}.
\end{definition}

\begin{lemma} \label{L:Dyadic intervals}
	Assume two dyadic intervals $\tau\sim \tau'$ with either $\tau,\tau' \subset \bR_{+}$ or $\tau,\tau' \subset \bR_{-}$. Then, we have the following properties:
	\begin{enumerate}[(1)]
		\item for any $\xi \in \tau$, there holds $|\xi| \leq 2|c(\tau)|$; \label{L:Dyadic intervals-1}
		\item there holds $|c(\tau')|\leq (2^{N_{\ell}+2}-1) |c(\tau)|$; \label{L:Dyadic intervals-2}
		\item for any $\xi\in\tau$ and $\xi'\in \tau'$, there holds $\frac{2^{N_{\ell}}}{2^{N_{\ell}}+1} |c(\tau+\tau')| \leq |\xi+\xi'| \leq \frac{2^{N_{\ell}}+2}{2^{N_{\ell}}+1}|c(\tau+\tau')|$; \label{L:Dyadic intervals-3}
		\item for any $\xi\in\tau$ and $\xi'\in\tau'$, there holds $2^{N_{\ell}-1} |\tau|\leq |\xi-\xi'| \leq 2^{N_{\ell}+1} |\tau|$. \label{L:Dyadic intervals-4}
	\end{enumerate}
\end{lemma}

To establish the desired refined Strichartz Proposition \ref{P:Refined Strichartz}, one key observations is the following geometric finite-cover property for the homogeneous odd curve $\xi^{\ell}$ case.
\begin{lemma}[Finite-cover property: odd curve] \label{L:Geometric quasi-orthogonal}
	For arbitrary odd integer $\ell>1$ and arbitrary pair $(\tau,\tau') \subset \bR_{+}^2$, denote the following bilinear Fourier support set
	\[S_{\ell}(\tau,\tau'):= \l\{\l(\xi+\xi', \xi^{\ell} +(\xi')^{\ell}\r): \xi\in \tau, \xi'\in \tau'\r\}.\]
	Then we can associate each interval $J \subset \bR_{+}$ with a parallelogram $R(J)\subset \bR^2$ such that
	\begin{equation} \label{E:Quasi-orthogonal-1}
		S_{\ell}(\tau,\tau') \subset R(\tau+\tau'), \quad \forall \tau\sim \tau';
	\end{equation}
	moreover, we can find a small constant $\beta_{\ell}>0$ and a big number $\widetilde{C}_{\ell} \in \bN$ which satisfy
	\begin{equation} \label{E:Quasi-orthogonal-2}
		\#\{(\tau_1,\tau'_1): \tau_1\sim \tau'_1, (1+\beta) R(\tau+\tau') \cap (1+\beta) R(\tau_1 +\tau'_1) \neq \emptyset\} \leq \widetilde{C}_{\ell}, \quad \forall \tau\sim\tau'.
	\end{equation}
	Here $(1+\beta) R(J)$ denotes the $(1+\beta)$-dilate of the parallelogram $R(J)$.
\end{lemma}

\begin{remark} 
	Given a fixed interval $\tau$, there always holds $c_{\tau} \sim_{\ell} c_{\tau'}$ for all intervals $\tau'\sim \tau$, see \eqref{E:Quasi-orthogonal-7} below. Thus for fixed $\tau$, the main idea for constructing the desired parallelogram $R(\tau+\tau')$ is to construct a linear function $K_{c_\tau,1} (\xi+\xi') + K_{c_\tau,0}$ such that
	\[\l|\l[\xi^{\ell} + (\xi')^{\ell} \r]- K_{c_\tau,1} (\xi+\xi') - K_{c_\tau,0}\r| \sim_{c_\tau} |\tau|^{\beta}, \quad \forall (\xi,\xi') \in \tau \times \tau' \quad \text{with} \quad \tau\sim \tau';\]
    and then define $R(\tau+\tau')$ as
    \[\l\{(\omega,\eta): \eta\in \tau+\tau', \;\; A_{c_{\tau},\ell} \leq \frac{\omega - K_{c_\tau,1} \eta - K_{c_\tau,0}}{|\tau|^{\beta}} \leq B_{c_{\tau},\ell}\r\}.\]
	In our proof below, we will take $\beta=2$ and then show the desired finite-cover property. 
\end{remark}

\begin{proof}[\textbf{Proof of Lemma \ref{L:Geometric quasi-orthogonal}}]
	Due to some technical reasons, we present an identity first. For a positive integer $k\in \bN$ and $(x,y) \in \bR^2$, one can check the following identity
	\begin{equation} \label{E:Identity}
		x^{2k+1}+y^{2k+1}-(x+y)^{2k+1}/2^{2k} =\frac{1}{2^{2k}} (x-y)^2(x+y) P_k(x,y),
	\end{equation}
	where the function $P_{k}(x,y)$ is defined by
	\[P_{k}(x,y):=\sum_{n=0}^{2k-2} b_n x^{2k-2-n} y^n,\]
	with the coefficients $b_0=2^{2k}-1$, $b_1=b_0-\binom{1}{2k+1}$ and
	\[b_{2i}=b_{2i-2}+b_0- \sum_{m=1}^{2i} \binom{m}{2k+1}, \quad b_{2i+1}=b_{2i}- \sum_{m=0}^i \binom{2m+1}{2k+1}, \quad i=1,2,\ldots, k.\]
	It is not hard to see that $b_n>0$ for all $n=1,2,\ldots, 2k+1$ and\footnote{Here $\lfloor k/2\rfloor$ denotes the floor function of the number $k/2$.}
	\[b_{2\lfloor k/2\rfloor}=\max \{b_n: n=1,2,\ldots, 2k+1\},\]
	which implies the following key estimate
	\[b_{2\lfloor k/2 \rfloor}^{-1} (x+y)^{2k-2} \leq P_k(x,y) \leq b_{2\lfloor k/2 \rfloor} (x+y)^{2k-2}, \quad \forall (x,y)\in \bR_{+}^2.\]
	Then we consider $\ell=2k+1$ and choose $N_{\ell}$ large enough such that
	\begin{equation} \label{E:Adjacent number}
		b_{2\lfloor (\ell-1)/4 \rfloor}^{-1} 2^{2N_{\ell}-\ell-3} \l(\frac{2^{N_{\ell}}}{2^{N_{\ell}}+1}\r)^{\ell-2} > 2^{\ell} \ell(\ell-1).
	\end{equation}
	This inequality can be achieved as long as $N_{\ell}$ is sufficiently large, and it is crucial for establishing the inequality \eqref{E:Quasi-orthogonal-8} below.

	Now let us construct the desired parallelogram which satisfies \eqref{E:Quasi-orthogonal-1}. By symmetry we may assume $|\xi'|< |\xi|$. For a parameter $\xi'$, we define the following function
	\[F_{\xi'}(\xi):=\xi^{\ell}+(\xi')^{\ell}-(\xi+\xi')^{\ell}/2^{\ell-1}.\]
	Based on the aforementioned identity \eqref{E:Identity} with $\ell=2k+1$, a direct computation shows
	\begin{equation*}
		\frac{1}{2^{\ell-1} b_{2\lfloor (\ell-1)/4 \rfloor}} (\xi-\xi')^2 (\xi+\xi')^{\ell-2} \leq F_{\xi'}(\xi) \leq \frac{b_{2\lfloor (\ell-1)/4 \rfloor}}{2^{\ell-1}} (\xi-\xi')^2(\xi+\xi')^{\ell-2}.
	\end{equation*}
	If $(\xi,\xi')\in \tau\times \tau'$ with $\tau\sim \tau'$ and $|\xi'|<|\xi|$, then the Lemma \ref{L:Dyadic intervals} further implies
	\begin{equation} \label{E:Quasi-orthogonal-3}
		F_{\xi'}(\xi) \geq b_{2\lfloor (\ell-1)/4 \rfloor}^{-1} 2^{2N_{\ell}-\ell-1} \l(\frac{2^{N_{\ell}}}{2^{N_{\ell}}+1}\r)^{\ell-2} |\tau|^2 c(\tau+\tau')^{\ell-2} 
	\end{equation}
	and
	\begin{equation} \label{E:Quasi-orthogonal-4}
		F_{\xi'}(\xi) \leq b_{2\lfloor (\ell-1)/4 \rfloor} 2^{2N_{\ell}-\ell+1} \l(\frac{2^{N_{\ell}}+2}{2^{N_{\ell}}+1}\r)^{\ell-2} |\tau|^2 c(\tau+\tau')^{\ell-2}.
	\end{equation}
	On the other hand, if we write
	\[\omega:=\xi^{\ell}+ (\xi')^{\ell}, \quad \eta:=\xi+\xi', \quad \bar{\tau}:=\tau+\tau', \quad c:=c(\bar{\tau})>0,\]
	then Taylor's theorem asserts the existence of $c_{\theta}$ between $c$ and $\eta$ such that
	\[|\eta|^{\ell-1}\eta =c^{\ell} +\ell c^{\ell-1}(\eta-c) +\frac{\ell(\ell-1)}{2} c_{\theta}^{\ell-2} |\eta-c|^2.\]
	Hence, by the estimates \eqref{E:Quasi-orthogonal-3} and \eqref{E:Quasi-orthogonal-4} we know that
	\[b_{2\lfloor (\ell-1)/4 \rfloor}^{-1} 2^{2N_{\ell}-\ell-1} \l(\frac{2^{N_{\ell}}}{2^{N_{\ell}}+1}\r)^{\ell-2} |\tau|^2 c^{\ell-2} \leq \omega-\frac{c^{\ell}+ \ell c^{\ell-1}(\eta-c)}{2^{\ell-1}} -\frac{\ell(\ell-1)c_{\theta}^{\ell-2} |\eta-c|^2}{2^{\ell}}\]
	and
	\[\omega-\frac{c^{\ell}+ \ell c^{\ell-1}(\eta-c)}{2^{\ell-1}} -\frac{\ell(\ell-1)c_{\theta}^{\ell-2} |\eta-c|^2}{2^{\ell}} \leq b_{2\lfloor (\ell-1)/4 \rfloor} 2^{2N_{\ell}-\ell+1} \l(\frac{2^{N_{\ell}}+2}{2^{N_{\ell}}+1}\r)^{\ell-2} |\tau|^2 c^{\ell-2}.\]
	Then we estimate the term $\omega-\frac{c^{\ell}+ \ell c^{\ell-1}(\eta-c)}{2^{\ell-1}}$, by Lemma \ref{L:Dyadic intervals} we conclude
	\begin{align} \label{E:Quasi-orthogonal-5}
		\omega-\frac{c^{\ell}+ \ell c^{\ell-1}(\eta-c)}{2^{\ell-1}} \geq b_{2\lfloor (\ell-1)/4 \rfloor}^{-1} 2^{2N_{\ell}-\ell-1} \l(\frac{2^{N_{\ell}}}{2^{N_{\ell}}+1}\r)^{\ell-2}  c^{\ell-2}|\tau|^2 =: A_{\ell}
	\end{align}
	and
	\begin{align} \label{E:Quasi-orthogonal-6}
		\omega-\frac{c^{\ell}+ \ell c^{\ell-1}(\eta-c)}{2^{\ell-1}} \leq \l[b_{2\lfloor (\ell-1)/4 \rfloor} 2^{2N_{\ell}-\ell+1} \l(\frac{2^{N_{\ell}}+2}{2^{N_{\ell}}+1}\r)^{\ell-2} + 3^{\ell-2} \ell(\ell-1)\r] c^{\ell-2} |\tau|^2 =:B_{\ell}.
	\end{align}
	Now the desired parallelogram can be given as follows. For any fixed interval $J \subset \bR_{+}$ with center $c(J)$ and length $|J|$, we define the linear function
    \begin{equation} \label{E:Quasi-orthogonal-6.5}
        G(\omega,\eta):= \frac{\omega-\frac{c(J)^{\ell}+ \ell c(J)^{\ell-1}[\eta-c(J)]}{2^{\ell-1}}}{c(J)^{\ell-2}|J|^2/4}
    \end{equation}
	and the associated parallelogram
	\[R(J):=\l\{(\omega,\eta): \eta\in J, A_{\ell} \leq G(\omega,\eta) \leq B_{\ell}\r\}.\]
	By the estimates \eqref{E:Quasi-orthogonal-5} and \eqref{E:Quasi-orthogonal-6}, we know that the parallelogram satisfies our condition \eqref{E:Quasi-orthogonal-1}.
	
	Finally, let us show the existence of $\beta_{\ell}$ and $\widetilde{C}_{\ell}$ such that the desired condition \eqref{E:Quasi-orthogonal-2} holds. Indeed, a direct computation shows $(1+\beta) R(\bar{\tau})$ is the following{\footnotesize
	\[\l\{(\omega,\eta): \eta\in \l[c_{\bar{\tau}}- \frac{(1+\beta)|\bar{\tau}|}{2}, c_{\bar{\tau}}+ \frac{(1+\beta)|\bar{\tau}|}{2} \r], A_{\ell}-(B_{\ell}-A_{\ell})\beta/2 \leq G(\omega,\eta) \leq B_{\ell}+ (B_{\ell}-A_{\ell})\beta/2\r\}.\]}
	Define $\bar{\tau}_1:=\tau_1+\tau'_1$ and let the point $(\widetilde{\omega},\widetilde{\eta})$ belong to the intersection $(1+\beta) R(\bar{\tau}) \cap (1+\beta) R(\bar{\tau}_1)$. Similar to the estimate \eqref{L:Dyadic intervals-3} in Lemma \ref{L:Dyadic intervals}, we have
	\[\frac{2^{N_{\ell}}-\beta}{2^{N_{\ell}}+1} \widetilde{c} \leq \eta \leq \frac{2^{N_{\ell}}+2+\beta}{2^{N_{\ell}}+1} \widetilde{c},\]
	where $\widetilde{c}$ denotes either $c(\bar{\tau})$ or $c(\bar{\tau}_1)$. Therefore, we have
	\begin{equation} \label{E:Quasi-orthogonal-7}
		\frac{2^{N_{\ell}}-\beta}{2^{N_{\ell}}+2+\beta} c(\bar{\tau}) \leq c(\bar{\tau}_1) \leq \frac{2^{N_{\ell}}+2+\beta}{2^{N_{\ell}}-\beta} c(\bar{\tau}).
	\end{equation}
	On the other hand, by Taylor's theorem, we can obtain
	\[\widetilde{\omega}-|\widetilde{\eta}|^{\ell-1}\widetilde{\eta}/2^{\ell-1} = \widetilde{\omega}-\frac{\widetilde{c}^{\ell}+\ell \widetilde{c}^{\ell-1} (\widetilde{\eta}-\widetilde{c})}{2^{\ell-1}} -\frac{\ell(\ell-1) \widetilde{c}_{\theta}^{\ell-2} (\widetilde{\eta}-\widetilde{c})^2}{2^{\ell}},\]
	where $\widetilde{c}_{\theta}$ is between $\widetilde{c}$ and $\widetilde{\eta}$. Let $\widetilde{\tau}$ denote $\bar{\tau}$ or $\bar{\tau}_1$. Then the definition of $(1+\beta) R(\widetilde{\tau})$ implies
	\begin{align*}
		\widetilde{\omega}-|\widetilde{\eta}|^{\ell-1}\widetilde{\eta}/2^{\ell-1} &\geq \l[A_{\ell}-\frac{\beta(B_{\ell}-A_{\ell})}{2}\r] \frac{\widetilde{c}^{\ell-2} |\widetilde{\tau}|^2}{4} -\frac{\ell(\ell-1)}{2^{\ell}} \Big(\widetilde{c}+(1+\beta)|\tau|\Big)^{\ell-2} (1+\beta)^2 |\tau|^2 \\
		&\geq \l[A_{\ell}- \frac{\beta(B_{\ell}-A_{\ell})}{2} - \frac{\ell (\ell-1)}{2^{\ell}} (3+2\beta)^{\ell-2} (1+\beta)^2\r] \widetilde{c}^{\ell-2} |\tau|^2
	\end{align*}
	and
	\begin{align*}
		\widetilde{\omega}-|\widetilde{\eta}|^{\ell-1}\widetilde{\eta}/2^{\ell-1} \leq \l[B_{\ell}+\frac{\beta(B_{\ell}-A_{\ell})}{2}\r] \widetilde{c}^{\ell-2} |\widetilde{\tau}|^2 \leq \l[4 B_{\ell}+ 2\beta(B_{\ell}-A_{\ell})\r] \widetilde{c}^{\ell-2} |\tau|^2.
	\end{align*}
	By the chosen of $N_{\ell}$ and the condition \eqref{E:Adjacent number}, there holds
	\begin{equation} \label{E:Quasi-orthogonal-8}
		A_{\ell}-2^{\ell} \ell(\ell-1)>0.
	\end{equation}
	Therefore, we can choose $\beta_{\ell}$ small enough such that
	\begin{equation} \label{E:Quasi-orthogonal-9}
		\frac{A_{\ell}}{9} \leq A_{\ell}- \frac{\beta_{\ell}(B_{\ell}-A_{\ell})}{2} - \frac{\ell (\ell-1)}{2^{\ell}} (3+2\beta_{\ell})^{\ell-2} (1+\beta_{\ell})^2,\quad 4 B_{\ell}+ 2\beta_{\ell} (B_{\ell}-A_{\ell}) \leq 5B_{\ell},
	\end{equation}
	which further deduces
	\[c(\bar{\tau})^{\ell-2} |\bar{\tau}|^2 \sim_{\ell} c(\bar{\tau}_1)^{\ell-2} |\bar{\tau}_1|^2\]
	and, moreover by \eqref{E:Quasi-orthogonal-7}, the relation $|\bar{\tau}| \sim_{\ell} |\bar{\tau}_1|$. This relation and the previous estimate \eqref{E:Quasi-orthogonal-7} imply the existence of our desired finite number $\widetilde{C}_{\ell}$ and complete the proof.
\end{proof}

\begin{remark}[Finite-cover property: inhomogeneous odd curve] \label{R:Geometric finite-cover}
	In our proof, the crucial fact is that the lower bound in the estimates \eqref{E:Quasi-orthogonal-9} is strictly greater than $0$, which gives the desired finite-cover property. Hence for the inhomogeneous odd curve $\xi^3+\xi^5$ case, one can similarly introduce the notations
	\[\omega:= \xi^3+\xi^5 + (\xi')^3 +(\xi')^5, \quad \eta := \xi + \xi',\]
	and imitate the linear function $G(\omega,\eta)$ in \eqref{E:Quasi-orthogonal-6.5} to define the following linear function
	\[G(\omega,\eta):= \frac{\omega-\frac{c(J)^3+ 3 c(J)^{2}[\eta-c(J)]}{4}- \frac{c(J)^5 + 5c(J)^4 [\eta-c(J)]}{16}}{[c(J) + c(J)^3]|J|^2/4}.\]
	Then by taking $N_5$ in \eqref{E:Adjacent number}, we conclude that this geometric finite-cover property also holds for $\xi^3+\xi^5$, since the lower bound in the estimates \eqref{E:Quasi-orthogonal-9} is still strictly greater than $0$ due to the following estimate
	\[\min\l\{\frac{A}{B},\frac{C}{D} \r\} \leq \frac{A+C}{B+D} \leq \max\l\{\frac{A}{B}, \frac{C}{D} \r\}, \quad \quad \forall (A,B,C,D) \in\bR_{+}^4.\]
\end{remark}
By this Remark \ref{R:Geometric finite-cover} and the quasi-orthogonal lemma of \cite[Lemma A.9]{KV2013}, see also \cite[Lemma 6.1]{TVV1998}, we directly obtain the following quasi-orthogonal consequence.
\begin{lemma}[Quasi-orthogonality] \label{L:Quasi-orthogonal}
	For any function $f\in L^2(\bR)$ with $\supp{f}\subset \bR_{-}$ or $\bR_{+}$, we have
	\[\l\|\sum_{\tau\sim \tau'} [\widetilde{E}_{5}]f_{\tau} [\widetilde{E}_{5}]f_{\tau'} \r\|_{L_{t,x}^3(\bR^2)}^{3/2} \lesssim \sum_{\tau\sim \tau'} \l\|[\widetilde{E}_{5}]f_{\tau} [\widetilde{E}_{5}]f_{\tau'} \r\|_{L_{t,x}^3(\bR^2)}^{3/2}.\]
\end{lemma}

The following bilinear estimate is classical by changing variables, the Hausdorff-Young inequality, and \holder inequality. Similar ideas can be seen in \cite[Lemma 2.1]{KPV1991-KDV} and \cite[Lemma 1.2]{Shao2009}. Here we show the details for the convenience of the reader.
\begin{lemma}[Bilinear estimate] \label{L:Bilinear Strichartz-3,5}
	Let $q\geq 2$. Then for all dyadic intervals $\tau\sim \tau'$ with either $\tau,\tau' \subset \bR_{+}$ or $\tau,\tau' \subset \bR_{-}$, and for any functions $u,v \in L^2(\bR)$ we have
	\begin{equation} \label{L:Bilinear Strichartz-3,5-1}
		\l\|[\widetilde{E}_{5}]u_{\tau} \cdot [\widetilde{E}_{5}] v_{\tau'}\r\|_{L_{t,x}^q(\bR^2)} \lesssim \l|c_{\tau} + c_{\tau}^3 \r|^{\frac{q-3}{3q}} |\tau|^{\frac{q-3}{q}} \|u\|_{L_x^2(\bR)} \|v\|_{L_x^2(\bR)}.
	\end{equation}
\end{lemma}

\begin{proof}[\textbf{Proof of Lemma \ref{L:Bilinear Strichartz-3,5}}]
	By direct computation, the function $[E_{\ell}]u_{\tau} \cdot [E_{\ell}] v_{\tau'}$ equals to
	\[\frac{1}{4\pi^2} \int_{\bR} \int_{\bR} e^{ix(\xi+\xi') +it[\xi^{3}+ \xi^{5} +(\xi')^{3}+ (\xi')^{5}]} |\xi+\xi^3|^{\frac{1}{6}} |\xi'+ (\xi')^3|^{\frac{1}{6}} \widehat{u}_{\tau}(\xi) \widehat{v}_{\tau'}(\xi') \ddd \xi \ddd \xi'.\]
	Then we can directly obtain the following $L^\infty$ estimate
	\[\l\|[\widetilde{E}_{5}]u_{\tau} \cdot [\widetilde{E}_{5}] v_{\tau'}\r\|_{L_{t,x}^{\infty}} \lesssim \l|c_{\tau} + c_{\tau}^3 \r|^{\frac{1}{3}} |\tau| \|u\|_{L_x^2} \|v\|_{L_x^2};\]
	on the other hand, by changing of variables, the Plancherel theorem and then returning to the original variables, we obtain the following $L^2$ estimate
	\[\l\|[\widetilde{E}_{5}]u_{\tau} \cdot [\widetilde{E}_{5}] v_{\tau'}\r\|_{L_{t,x}^2} \lesssim \l|c_{\tau} + c_{\tau}^3 \r|^{-\frac{1}{6}} |\tau|^{-\frac{1}{2}} \|u\|_{L_x^2} \|v\|_{L_x^2}.\]
	Finally, we obtain the desired consequence \eqref{L:Bilinear Strichartz-3,5-1} by interpolation.
\end{proof}

With the bilinear estimate Lemma \ref{L:Bilinear Strichartz-3,5} and quasi-orthogonal Lemma \ref{L:Quasi-orthogonal} in place, now we are ready to state our final result in this subsection, the refined Strichartz estimate Proposition \ref{P:Refined Strichartz}. This consequence follows from a standard Whitney-type decomposition argument. We omit its proof to avoid too much repetition. Details can be seen, for instance, in the proofs of \cite[Proposition 4.24]{KV2013} and \cite[Lemma 3.7]{FS2018}.
\begin{prop}[Refined Strichartz: inhomogeneous curve] \label{P:Refined Strichartz}
	There exists $\theta\in (0,1)$ such that for all $f\in L^2(\bR)$ we have
	\begin{equation*}
		\l\|[\widetilde{E}_{5}] f \r\|_{L_{t,x}^6(\bR^2)} \lesssim \l[\sup_{\tau\in\mathcal{D}} |c_{\tau} + c_{\tau}^3|^{-\frac{1}{6}} |\tau|^{-\frac{1}{2}} \l\|[\widetilde{E}_{5}] f_{\tau}\r\|_{L_{t,x}^{\infty}(\bR^2)} \r]^{\theta} \|f\|_{L_x^2(\bR)}^{1-\theta}.
	\end{equation*}
\end{prop}

\subsection{Profile decomposition} \label{SS:Profile decomposition-3,5}
To establish the $\xi^3+\xi^5$ version profile decomposition, our strategy is the following: first establishing the $|\xi|^3+|\xi|^5$ version profile decomposition, and then establishing the desired $\xi^3+\xi^5$ version profile decomposition. For the first step, we use the tools established in the previous two subsections; for the second step, we use the divide-and-regroup arguments developed in the previous section.

The classical profile decomposition consequences, with only one profile, are well-used and well-understood in the literature. Indeed, it is a standard argument that applying the approximate local convergence Lemma \ref{L:Approximate local convergence}, the Br\'ezis-Lieb type lemma \cite[Theorem 1.9]{LL2001}, and the refined Strichartz estimate Proposition \ref{P:Refined Strichartz} to establish the corresponding $|\xi|^3+|\xi|^5$ version profile decomposition. Hence, for simplicity, the arguments in this subsection will be sketchy, and the readers are referred to the classical literature \cite{BG1999,BV2007,CK2007,Keraani2001,KV2013} for further details.

At first, we introduce the associated even curve $|\xi|^3+|\xi|^5$ extension operator
\[[\widetilde{E}_{*,5}]f(x):=\frac{1}{2\pi} \int_{\bR} |\xi + \xi^3|^{\frac{1}{6}} e^{ix\xi+it(|\xi|^3+|\xi|^5)} \widehat{f}(\xi) \ddd \xi,\]
and for parameters $(h_n^j, x_n^j, \xi_n^j, t_n^j)$ we define the corresponding profile operator
\begin{equation} \label{E:Profile operator-3,5-e}
	[\widetilde{T}_{*,n}^j]f(x):= \l[e^{it(|\nabla|^3+|\nabla|^5)}\r] \l((h_n^j)^{-\frac{1}{2}} e^{i(x-x_n^j)\xi_n^j} \phi_1 \l(\frac{x-x_n^j}{h_n^j}\r)\r).
\end{equation}

\begin{definition}[Orthogonal parameters: inhomogeneous even curve]
	We say the parameters $(h_n^j, x_n^j, \xi_n^j, t_n^j)$ are \textit{inhomogeneous even pairwise orthogonal} if for arbitrary fixed $j\neq k$, these parameters satisfy one of the following:
	\begin{itemize}
		\item $(h_n^j,\xi_n^j) \neq (h_n^k,\xi_n^k)$ with
		\[\lim_{n\to \infty}\l(\frac{h_n^j}{h_n^k} + \frac{h_n^k}{h_n^j} + (h_n^j+h_n^k) |\xi_n^j-\xi_n^k| \r) = \infty;\]
		\item $(h_n^j,\xi_n^j) \equiv (h_n^k,\xi_n^k) \equiv (h_n, \xi_n)$ with $|h_n\xi_n| \to \infty$ and
		\[\lim_{n\to\infty} \l(\l|\frac{x_n^j-x_n^k}{h_n} -\frac{3(t_n^j-t_n^k) (h_n \xi_n)^{2}}{(h_n)^{3}} - \frac{5(t_n^j-t_n^k) (h_n \xi_n)^{4}}{(h_n)^{5}}\r| + \l|\frac{(t_n^j-t_n^k) (h_n \xi_n)^{3}}{(h_n)^{5}}\r| \r)= \infty;\]
		\item $(h_n^j,\xi_n^j) \equiv (h_n^k,\xi_n^k) \equiv (h_n, \xi_n)$ with $\xi_n \equiv 0$ and
		\[\lim_{n\to\infty} \l(\l|\frac{x_n^j-x_n^k}{h_n} \r| + \l|\frac{t_n^j-t_n^k}{(h_n)^{3}}\r| + \l|\frac{t_n^j-t_n^k}{(h_n)^{5}}\r|\r)= \infty.\]
	\end{itemize}
\end{definition}

\begin{thm}[Profile decomposition: inhomogeneous even curve]
	Let $(f_n)$ be a bounded sequence in $L^2(\bR)$. Then, up to subsequences, there exist a sequence of operators $([\widetilde{T}_{*,n}^j])$ defined by \eqref{E:Profile operator-3,5-e} with inhomogeneous even pairwise orthogonal parameters $(h_n^j, x_n^j, \xi_n^j, t_n^j)$ and a sequence of functions $(\phi^j) \subset L^2(\bR)$ such that for every integer $J\in \bN$, we have the profile decomposition
	\begin{equation*}
		f_n=\sum_{j=1}^{J} [\widetilde{T}_{*,n}^j]\phi^j +\omega_n^{J},
	\end{equation*}
	where this decomposition possesses the following properties: firstly, the remainder term $\omega_n^{J}$ has vanishing Strichartz norm
	\begin{equation*}
		\lim_{J\to\infty}\limsup_{n\to\infty}\l\|[\widetilde{E}_{*,5}] \omega_n^{J} \r\|_{L_{t,x}^{6}}=0;
	\end{equation*}
	secondly, for each $j\leq J$, there holds the following $L_x^2$ weak convergence
	\begin{equation*}
		[\widetilde{T}_{*,n}^j]^{-1} \omega_n^J \rightharpoonup 0 \quad \text{as} \quad n\to \infty;
	\end{equation*}
	thirdly, for each $J\geq1$, there holds the $L^2$-limit orthogonality
	\begin{equation*}
		\lim_{n\to\infty}\l[\|f_n\|_{L_x^2}^2-\l(\sum_{j=1}^{J} \l\|[\widetilde{T}_{*,n}^j]\phi^j \r\|_{L_x^2}^2\r)-\|\omega_n^{J}\|_{L_x^2}^2\r]=0;
	\end{equation*}
	moreover, for arbitrary $j\neq k$, there holds
	\begin{equation*}
		\lim_{n\to\infty} \l\|[\widetilde{E}_{*,5}] [\widetilde{T}_{*,n}^j]\phi^j \cdot [\widetilde{E}_{*,5}] [\widetilde{T}_{*,n}^k]\phi^k \r\|_{L_{t,x}^3}=0,
	\end{equation*}
	which further implies that, for each $J\geq 1$, there holds the Strichartz-limit orthogonality
	\begin{equation*}
        \limsup_{n\to\infty}\l(\l\|\sum_{j=1}^J[\widetilde{E}_{*,5}] [\widetilde{T}_{*,n}^j]\phi^j\r\|_{L_{t,x}^6}^{6} -\sum_{j=1}^J\l\|[\widetilde{E}_{*,5}] [\widetilde{T}_{*,n}^j]\phi^j\r\|_{L_{t,x}^6}^{6}\r)=0.
	\end{equation*}
\end{thm}

Finally, by following the divide-and-regroup argument shown in Section \ref{S:Homogeneous odd curve}, one can similarly obtain the desired inhomogeneous odd curve version profile decomposition Theorem \ref{T:Profile decomposition-3,5}.

\subsection*{Acknowledgments}
The authors would like to thank Ryan Frier and Shuanglin Shao for their valuable conversations. The first author acknowledges the support from the University of Chinese Academy of Sciences Joint Training Program. Part of this work was completed during the first author's visit to the University of Kansas whose hospitality is also appreciated.

\bigskip\bigskip
\begin{flushleft}
	\vspace{0.3cm}\textsc{Boning Di \\
		School of Mathematical Sciences, University of Chinese Academy of Sciences, Beijing, 100049, People's Republic of China \\
		Academy of Mathematics and Systems Science, Chinese Academy of Sciences, Beijing, 100190, People's Republic of China} \\
		\emph{E-mail address}: \textsf{diboning18@mails.ucas.ac.cn}
		
	\vspace{0.3cm}\textsc{Chenjie Fan \\
		Academy of Mathematics and Systems Science, Chinese Academy of Sciences, Beijing, 100190, People's Republic of China \\
		Hua Loo-Keng Key Laboratory of Mathematics, Chinese Academy of Sciences, Beijing, 100190, People's Republic of China} \\
		\emph{E-mail address}: \textsf{cjfanpku@gmail.com}
	
	\vspace{0.3cm}\textsc{Dunyan Yan\\
		School of Mathematical Sciences, University of Chinese Academy of Sciences, Beijing, 100049, People's Republic of China} \\
		\emph{E-mail address}: \textsf{ydunyan@ucas.ac.cn}
\end{flushleft}

\end{document}